%% file: GRLW_PET_GAL_v3_oct_16_2018.tex
\documentclass[12pt]{article}
\usepackage{eurosym}
\usepackage{epsfig}
\usepackage{graphicx}
\usepackage{color}
\usepackage{amsmath,amsfonts,amssymb,amsthm}
\usepackage{fullpage,epsfig}

\newtheorem{thrm}{Theorem}

\newfont{\afont}{cmr12 at 8pt}
\newfont{\bfont}{cmr12 at 7pt}
in \evensidemargin 0.25true in \topmargin -0.2true in \headsep
0.4true in
\def\vec#1{\leqavevmode\vtop{\hbox to 1em{\hss$#1$\hss}
\vskip-\baselineskip\vskip1.1ex\hbox to
1em{\hfil$\scriptstyle\sim$\hfil}}}
\begin{document}

\date{ }
\date{ }
\title{ Numerical approximation of the generalized regularized long wave equation using Petrov-Galerkin
finite element method }
\author{Seydi Battal Gazi Karakoc$^1$ and Samir Kumar Bhowmik$^2$ ~\footnote{corresponding author, e-mail:  bhowmiksk@gmail.com}\\
\\
$1.$ Department of Mathematics, Faculty of Science and Art,\\
Nevsehir Haci Bektas Veli University, Nevsehir, 50300, Turkey.\\
e-mail: sbgkarakoc@nevsehir.edu.tr\\
$2.$ Department of Mathematics, University of Dhaka, \\
1000 Dhaka, Bangladesh. \\
e-mail: bhowmiksk@gmail.com }
\maketitle

\begin{abstract}
The generalized regularized long wave (GRLW) equation has been developed to
model a variety of physical phenomena such as ion-acoustic and
magnetohydrodynamic waves in plasma, nonlinear transverse waves in shallow
water and phonon packets in nonlinear crystals. This paper aims to develop
and analyze a powerful numerical scheme for the nonlinear generalized
regularized long wave (GRLW) equation by Petrov--Galerkin method in which
the element shape functions are cubic and weight functions are quadratic
B-splines. The suggested method is performed to three test problems
involving propagation of the single solitary wave, interaction of two
solitary waves and evolution of solitons with the Maxwellian initial
condition. The variational formulation and semi-discrete Galerkin scheme of
the equation are firstly constituted. We estimate accuracy of such a spatial
approximation. Then Fourier stability analysis of the linearized scheme
shows that it is unconditionally stable. To verify practicality and
robustness of the new scheme error norms $L_{2}$, $L_{\infty }$ and three
invariants $I_{1},I_{2}$ and $I_{3}$ are calculated. The obtained numerical
results are compared with other published results and shown to be precise
and effective.
\end{abstract}

%





\section{Introduction}

The GRLW equation was originated by a famous nonlinear analyst Peregrine who
first successfully introduced the regularized long wave equation as a
perfect alternative to the famous KdV equation for studying soliton
phenomena and as a mathematical model for small amplitude long waves on the
surface of water ~\cite{pereg}. Nonlinear evolution equations play
fundemental roles in various fields of science mostly in physics, applied
mathematics and in engineering problems. Analytical solutions of these
equations are commonly not derivable, particularly when the nonlinear terms
are contained. Numerical solutions of these equations are very practical to
analyze the physical phenomena due to the fact that analytical solutions of
these equations are found for the restricted boundary and initial
conditions. The regularized long wave (RLW) equation

\begin{equation}
u_{t}+u_{x}+auu_{x}-bu_{xxt}=0,  \label{rlw}
\end{equation}%
is one of the important model in physics media on account of it defines
phenomena with weak nonlinearity and dispersion waves, involving nonlinear
transverse waves in shallow water, ion-acoustic waves in plasma,
hydromagnetic wave in cold plasma, plasma, elastic media, optical fibres,
acoustic-gravity waves in compressible fluids, pressure waves in liquid--gas
bubbles and acoustic waves inharmonic crystals. The solutions of this
equation are sorts of solitary waves called as solitons whose form are not
affected by a collision. It was first alleged by Peregrine \cite%
{pereg,pereg1} for studying soliton phenomena and as a sample for
small-amplitude long-waves on the surface of water in a channel and widely
studied by Benjamin et al. \cite{ben}. In physical situations such as
unidirectional waves propagating in a water channel, long-crested waves in
near-shore zones, and many others, the RLW equation serves as an alternative
model to the Korteweg--de Vries equation (KdV equation) \cite{bona1,bona2}.
An exact solution of the equation was obtained under the limited initial and
boundary conditions in \cite{bona} for this reason it got fascinate from a
numerical point of view. Therefore, numerical solutions of the RLW equation
have been the matter of many papers. Various effective methods have been
presented to solve the equation such as finite difference method [7--10],
pseudo-spectral method \cite{Gou} , meshfree method \cite{sir}, Adomian
decomposition method \cite{kaya2} and various forms of finite element
methods in [14--24] . Indeed, the RLW equation is a special case of the
generalized regularized long wave (GRLW) equation which is an alternative to
the KdV equation for describing nonlinear dispersive waves and can be used
to characterise phenomena with weak nonlinearity and dispersion waves. It is
defined as

\begin{equation}
u_{t}+u_{x}+p(p+1)u^{p}u_{x}-\mu u_{xxt}=0,~~~~~  \label{grlw}
\end{equation}%
subject to the initial condition
\begin{equation}
u(x,0)=f_1(x),~~~~~a\leq x\leq b,  \label{intl}
\end{equation}%
and the boundary conditions
\begin{equation}
\begin{array}{lll}
u(a,t)=0,~~~~~u(b,t)=0, &  &  \\
u_{x}(a,t)=0,~~~~~u_{x}(b,t)=0, &  &  \\
u_{xx}(a,t)=0,~~~~~u_{xx}(b,t)=0, & t>0 &
\end{array}
\label{bndry}
\end{equation}%
where $p$ is a positive integer, $%
\mu
$ is positive constant and physical boundary conditions require $u$ and $%
u_{x}$ $\rightarrow 0$ that $u\rightarrow 0$ for $x\rightarrow \pm \infty $.
In equation $(\ref{grlw})$ $u$ indicates dimensionless surface elevation, $x$
distance and $t$ time. On the other hand, the GRLW equation has received
much less attention, presumably because of its higher nonlinearity for $p>2$
and the fact that it possesses a finite number of conserved quantities and
admits solitary waves as solutions, but, unlike other equations, the
stability of its solutions depends on their velocity \cite{bona3}. Some
solitary wave solutions for GRLW equations have been obtained by Hamdi et
al. \cite{hamdi} and Ramos \cite{ramos} studied solitary wave interactions
based on the separation of the temporal and spatial derivatives. Zhang \cite%
{zhang} implemented finite difference method for a Cauchy problem while Kaya
\cite{kaya}, Kaya and El-Sayed \cite{kaya1} indicated the numerical solution
of the GRLW equation by using the Adomian decomposition method. Roshan \cite%
{roshan} have procured numerical solutions of the GRLW equation by the
application of Petrov--Galerkin method, which uses a linear hat function as
the trial function and a quintic B-spline function as the test function.
Wang et al. \cite{wang} offered a mesh-free method for the GRLW equation
based on the moving least-squares approximation. Karako\c{c} \cite{karakoc}
and Zeybek \cite{karakoc1} have obtained solitary-wave solutions of the GRLW
equation by using septic B-spline collocation and cubic B-spline lumped
Galerkin method.\bigskip\ Numerical solutions of the GRLW \ \ \ equation
have been obtained by Soliman \cite{soliman1} using He's variational
iteration method. Mokhtari and Mohammadi \cite{mohammadi} suggested the
Sinc-collocation method for this equation. A time-linearization method that
uses a Crank--Nicolson procedure in time and three point, fourth-order
accurate in space, compact difference equations, is presented and used to
determine the solutions of the GRLW equation and a modified version thereof
(mGRLW) by C.M. Garc\'{\i}a-L\'{o}pez and J. I. Ramos \cite{garcia}. The
another special case of the GRLW equation is the modified regularized long
wave (MRLW) equation for $p=2$. MRLW equation was solved numerically by
various methods \cite{gard3,khalifa,ras,hag,dag5,krkc,krkc1}.

Spline functions are a class of piecewise polynomials which provide
continuity features being subject to the degree of the polynomials. They are
spectacular mathematical instrument for numerical approximations because of
their numerous popular specialities. One kind of splines, noted as
B-splines, has been used in obtaining the numerical solution of the GRLW
equation \cite{karakoc,karakoc1,khalifa,ras}. Assemblies of B-splines are
used as trial functions in the Petrov-Galerkin methods. Especially, cubic
B-splines associated with finite element methods have been verified to give
very smooth solutions, and use of the cubic B-splines as shape functions in
the finite element method warranties continuity of the first and
second-order derivatives of trial solutions at the mesh points \cite{dag}.

In this study, we have designed a lumped Petrov-Galerkin method for the GRLW
equation using cubic B-spline function as element shape function and
quadratic B-spline function as the weight function. The plan of this paper
is as follows:

\begin{itemize}
\item In Section 2, the governing equation and its variational formulation
and newly established theorems are presented.

\item A semi-discrete Galerkin scheme of the equation is notified in Section
3.

\item In Section 4, a lumped Petrov-Galerkin finite element technique has
been practiced to GRLW equation. Resulting system can be solved with a sort
of the Thomas algorithm.

\item Section 5, is dedicated to stability analysis of the method.

\item The results of numerical examples are reported in Section 6. The last
section is a brief conclusion.
\end{itemize}

\input{analysis_v1.tex}

\section{Numerical implementations of the scheme}

Let us consider the solution domain is limited to a finite interval $a\leq
x\leq b.$ Partition the interval $[a,b]$ at points by $x_{m}$ where $%
a=x_{0}<x_{1}<...<x_{N}=b$ and let $h=\frac{b-a}{N},$ $m=0,1,2,...,N$. On
this partition, we shall need the following cubic B-splines $\phi _{m}(x)$
at the points $x_{m},$ $m=0,1,2,...,N.$ The cubic B-spline functions $\phi
_{m}(x)$, (\emph{m}= $-1(1)$ $N+1$) are identified as follows \cite{prenter}
\begin{equation}
\begin{array}{l}
\phi _{m}(x)=\frac{1}{h^{3}}\left\{
\begin{array}{ll}
(x-x_{m-2})^{3},~~ & ~x\in \lbrack x_{m-2},x_{m-1}), \\
{h^{3}}+3{h^{2}}(x-x_{m-1})+3h(x-x_{m-1})^{2}-3(x-x_{m-1})^{3},~~ & ~x\in
\lbrack x_{m-1},x_{m}), \\
{h^{3}}+3{h^{2}}(x_{m+1}-x)+3h(x_{m+1}-x)^{2}-3(x_{m+1}-x)^{3},~~ & ~x\in
\lbrack x_{m},x_{m+1}), \\
(x_{m+2}-x)^{3},~~ & ~x\in \lbrack x_{m+1},x_{m+2}], \\
0~ & ~otherwise.%
\end{array}%
\right.
\end{array}
\label{3}
\end{equation}%
We search the approximation solution $u_{N}(x,t)$ to the\ exact solution $%
u(x,t)$ which uses these cubic B-splines as trial functions
\begin{equation}
u_{N}(x,t)=\sum_{j=-1}^{N+1}\phi _{j}(x)\delta _{j}(t),\   \label{4}
\end{equation}%
where $\delta _{j}(t)$ are time depended quantities or the nodal parameters
to be detected from boundary and weighted residual conditions. Applying the
following transformation

\begin{equation}
h\eta =x-x_{m}\text{ \ \ \ \ }0\leq \eta \leq 1,  \label{250}
\end{equation}%
the finite interval $[x_{m},x_{m+1}]$ is turned into more easily practicable
interval $[0,1]$. Therefore cubic B-spline shape functions $(\ref{3})$
depending on variable $\eta $ on the region $[0,1]$ rearranged with
\begin{equation}
\begin{array}{l}
\phi _{m-1}=(1-\eta )^{3}, \\
\phi _{m}=1+3(1-\eta )+3(1-\eta )^{2}-3(1-\eta )^{3}, \\
\phi _{m+1}=1+3\eta +3\eta ^{2}-3\eta ^{3}, \\
\phi _{m+2}=\eta ^{3}.%
\end{array}
\label{5}
\end{equation}%
All splines, apart from $\phi _{m-1}(x),\phi _{m}(x),\phi _{m+1}(x)$,$\phi
_{m+2}(x)$ and their four principal derivatives are null over the region $%
[0,1].$ Thereby variation of the function $u(x,t)$ over element $[0,1]$ is
approximated by
\begin{equation}
u_{N}(\eta ,t)=\sum_{j=m-1}^{m+2}\delta _{j}\phi _{j},  \label{6}
\end{equation}%
where $\delta _{m-1},\delta _{m},\delta _{m+1},\delta _{m+2}$ and B-spline
element functions $\phi _{m-1},\phi _{m},\phi _{m+1},\phi _{m+2}$ as element
shape functions. The nodal values $u$ and its derivatives up to second order
at the knots $x_{m}$ are given in terms of the parameters $\delta _{m}$ from
the use of the B-splines $(\ref{5})$ and\ and the trial solution $(\ref{6})$%
:
\begin{equation}
\begin{array}{l}
u_{m}=u(x_{m})=\delta _{m-1}+4\delta _{m}+\delta _{m+1}, \\
u_{m}^{\prime }=u^{\prime }(x_{m})=3(-\delta _{m-1}+\delta _{m+1}), \\
u_{m}^{\prime \prime }=u^{\prime \prime }(x_{m})=6(\delta _{m-1}-2\delta
_{m}+\delta _{m+1}).%
\end{array}
\label{7}
\end{equation}%
We take the weight functions $\Phi _{m}$ as quadratic B-splines. The
quadratic B-splines $\Phi _{m}$ at the knots $x_{m}$ are defined as \cite%
{prenter}:

\begin{equation}
\begin{array}{l}
\Phi _{m}(x)=\frac{1}{h^{2}}\left\{
\begin{array}{ll}
(x_{m+2}-x)^{2}-3(x_{m+1}-x)^{2}+3(x_{m}-x)^{2},~~ & ~x\in \lbrack
x_{m-1},x_{m}), \\
(x_{m+2}-x)^{2}-3(x_{m+1}-x)^{2},~~ & ~x\in \lbrack x_{m},x_{m+1}), \\
(x_{m+2}-x)^{2},~~ & ~x\in \lbrack x_{m+1},x_{m+2}), \\
0~ & ~otherwise.%
\end{array}%
\right.%
\end{array}
\label{quad}
\end{equation}%
When we take into consideration of the transformation $(\ref{250})$,
quadratic B-splines $\Phi _{m}$ are written as

\begin{equation}
\begin{array}{l}
\Phi _{m-1}=(1-\eta )^{2}, \\
\Phi _{m}=1+2\eta -2\eta ^{2}, \\
\Phi _{m+1}=\eta ^{2}.%
\end{array}%
\end{equation}%
Performing the Petrov-Galerkin method to Eq.$(\ref{grlw}),$ the weak form of
Eq.$(\ref{grlw})$ is attained as

\begin{equation}
\int_{a}^{b}\Phi (u_{t}+u_{x}+p(p+1)u^{p}u_{x}-\mu u_{xxt})dx=0.  \label{8}
\end{equation}%
For a unique element $[x_{m},x_{m+1}]$ using transformation $(\ref{250})$
into Eq.$(\ref{8}),$ we obtain the following integral equation:

\begin{equation}
\int_{0}^{1}\Phi \left( u_{t}+\frac{1}{^{h}}u_{\eta }+\frac{p(p+1)}{h}\hat{u}%
^{p}u_{\eta }-\frac{\mu }{h^{2}}u_{\eta \eta t}\right) d\eta =0,  \label{9}
\end{equation}%
where $\hat{u}$ is accepted to be constant over an element to ease the
integral. Integrating Eq.$(\ref{9})$ by parts and using Eq.$(\ref{grlw})$
which yields:

\begin{equation}
\int_{0}^{1}[\Phi (u_{t}+\lambda u_{\eta })+\beta \Phi _{\eta }u_{\eta
t}]d\eta =\beta \Phi u_{\eta t}|_{0}^{1},  \label{100}
\end{equation}%
where $\lambda =\frac{1+p(p+1)\hat{u}^{p}}{h}$ and $\beta =\frac{\mu }{h^{2}}%
.$ Assuming the weight function $\Phi _{i}$ with quadratic B-spline shape
functions given by Eq.$(\ref{quad})$ and substituting approximation $(\ref{6}%
)$ into integral Eq.$(\ref{100})$, we get the element contributions in the
form:

\begin{equation}
\sum_{j=m-1}^{m+2}[(\int_{0}^{1}\Phi _{i}\phi _{j}+\beta \Phi _{i}^{\prime
}\phi _{j}^{\prime })d\eta -\beta \Phi _{i}\phi _{j}^{\prime }|_{0}^{1}~~]%
\dot{\delta}_{j}^{e}+\sum_{j=m-1}^{m+2}(\lambda \int_{0}^{1}\Phi _{i}\phi
_{j}^{\prime }d\eta )\delta _{j}^{e}=0,  \label{11}
\end{equation}%
where $\delta ^{e}=(\delta _{m-1},\delta _{m},\delta _{m+1},\delta
_{m+2})^{T}$ are the element parameters and dot states differentiation to $t$
which can be written in matrix form as follows:

\begin{equation}
\lbrack A^{e}+\beta (B^{e}-C^{e})]\dot{\delta}^{e}+\lambda D^{e}\delta
^{e}=0.  \label{120}
\end{equation}%
The element matrices $A_{ij}^{e},B_{ij}^{e},C_{ij}^{e}$ and $D_{ij}^{e}$ are
rectangular $3\times 4$ given by the following integrals;
\begin{equation*}
A_{ij}^{e}=\int_{0}^{1}\Phi _{i}\phi _{j}d\eta =\frac{1}{60}\left[
\begin{array}{cccc}
10 & 71 & 38 & 1 \\
19 & 221 & 221 & 19 \\
1 & 28 & 71 & 10%
\end{array}%
\right] ,
\end{equation*}

\begin{equation*}
B_{ij}^{e}=\int_{0}^{1}\Phi _{i}^{\prime }\phi _{j}^{\prime }d\eta =\frac{1}{%
2}\left[
\begin{array}{cccc}
3 & 5 & -7 & -1 \\
-2 & 2 & 2 & -2 \\
-1 & -7 & 5 & 3%
\end{array}%
\right] ,
\end{equation*}

\begin{equation*}
C_{ij}^{e}=\Phi _{i}\phi _{j}^{\prime }|_{0}^{1}=3\left[
\begin{array}{cccc}
1 & 0 & -1 & 0 \\
1 & -1 & -1 & 1 \\
0 & -1 & 0 & 1%
\end{array}%
\right] ,
\end{equation*}

\begin{equation*}
D_{ij}^{e}=\int_{0}^{1}\Phi _{i}\phi _{j}^{\prime }d\eta =\frac{1}{10}\left[
\begin{array}{cccc}
-6 & -7 & 12 & 1 \\
-13 & -41 & 41 & 13 \\
-1 & -12 & 7 & 6%
\end{array}%
\right]
\end{equation*}%
where $i$ takes the values $1,2,3$ and the $j$ takes the values $%
m-1,m,m+1,m+2$ for the typical element $[x_{m},x_{m+1}].$ A lumped value for
$u$ is obtained from $\left( u_{m}+u_{m+1}\right) ^{2}/4$ as
\begin{equation*}
\lambda =\frac{1}{4h}\left( \delta _{m-1}+5\delta _{m}+5\delta _{m+1}+\delta
_{m+2}\right) ^{2}.
\end{equation*}%
Combining contributions from all elements induces to the\ following matrix
equations
\begin{equation}
\lbrack A+\beta (B-C)]\dot{\delta}+\lambda D\delta =0,  \label{13}
\end{equation}%
where $\delta =(\delta _{-1},\delta _{0},...,\delta _{N},\delta _{N+1})^{T}$
global element parameters. The $A$, $B$, $C$ and $\lambda D$ matrices are
rectangular and row $m$ of each has the following form:
\begin{equation*}
\begin{array}{l}
A=\frac{1}{60}\left( 1,57,302,302,57,1,0\right) ,\text{ }B=\frac{1}{2}%
(-1,-9,10,10,-9,-1,0), \\
C=(0,0,0,0,0,0,0) \\
\lambda D=\frac{1}{10}\left(
\begin{array}{c}
-\lambda _{1},-12\lambda _{1}-13\lambda _{2},7\lambda _{1}-41\lambda
_{2}-6\lambda _{3},6\lambda _{1}+41\lambda _{2}-7\lambda _{3}, \\
13\lambda _{2}+12\lambda _{3},\lambda _{3}%
\end{array}%
\right)
\end{array}%
\end{equation*}%
where
\begin{equation*}
\begin{array}{l}
\lambda _{1}=\frac{1}{4h}\left( \delta _{m-2}+5\delta _{m-1}+5\delta
_{m}+\delta _{m+1}\right) ^{2},\ \lambda _{2}=\frac{1}{4h}\left( \delta
_{m-1}+5\delta _{m}+5\delta _{m+1}+\delta _{m+2}\right) ^{2}, \\
\lambda _{3}=\frac{1}{4h}\left( \delta _{m}+5\delta _{m+1}+5\delta
_{m+2}+\delta _{m+3}\right) ^{2}.%
\end{array}%
\end{equation*}%
Replacing the time derivative $\dot{\delta}$ by the forward difference
approximation $\dot{\delta}=\frac{\delta ^{n+1}-\delta ^{n}}{\Delta t}$ and
the parameter $\delta $ by the Crank-Nicolson formulation $\delta =\frac{1}{2%
}(\delta ^{n}+\delta ^{n+1}),$ then Eq. $(\ref{13})$ reduce to the following
matrix system:
\begin{equation}
\lbrack A+\beta (B-C)+\frac{\lambda \Delta t}{2}D]\delta ^{n+1}=[A+\beta
(B-C)-\frac{\lambda \Delta t}{2}D]\delta ^{n}  \label{14}
\end{equation}%
where $t$ is time step. Applying the boundary conditions ($\ref{bndry})$ to
the system $(\ref{14})$, $(N+1)\times (N+1)$ matrix system is obtained. This
last system is actively solved with a variant of the Thomas algorithm but in
solution process, two or three inner iterations $\delta ^{n\ast }=\delta
^{n}+\frac{1}{2}(\delta ^{n}-\delta ^{n-1})$ are also practiced at each time
step to overcome the non-linearity. Ultimately, a typical member of the
matrix system $(\ref{14})$ \ is written in terms of the nodal parameters $%
\delta ^{n}$ and $\delta ^{n+1}$ as:
\begin{equation}
\begin{array}{l}
\gamma _{1}\delta _{m-2}^{n+1}+\gamma _{2}\delta _{m-1}^{n+1}+\gamma
_{3}\delta _{m}^{n+1}+\gamma _{4}\delta _{m+1}^{n+1}+\gamma _{5}\delta
_{m+2}^{n+1}+\gamma _{6}\delta _{m+3}^{n+1}= \\
\gamma _{6}\delta _{m-2}^{n}+\gamma _{5}\delta _{m-1}^{n}+\gamma _{4}\delta
_{m}^{n}+\gamma _{3}\delta _{m+1}^{n}+\gamma _{2}\delta _{m+2}^{n}+\gamma
_{1}\delta _{m+3}^{n},%
\end{array}
\label{15}
\end{equation}%
where
\begin{equation*}
\begin{array}{l}
\gamma _{1}=\frac{1}{60}-\frac{\beta }{2}-\frac{\lambda \Delta t}{20},~~~\ \
\ \ \ \ \ \gamma _{2}=\frac{57}{60}-\frac{9\beta }{2}-\frac{25\lambda \Delta
t}{20},\gamma _{3}=\frac{302}{60}+\frac{10\beta }{2}-\frac{40\lambda \Delta t%
}{20}, \\
\gamma _{4}=\frac{302}{60}+\frac{10\beta }{2}+\frac{40\lambda \Delta t}{20}%
,~~~\ \gamma _{5}=\frac{57}{60}-\frac{9\beta }{2}+\frac{25\lambda \Delta t}{%
20},\gamma _{6}=\frac{1}{60}-\frac{\beta }{2}+\frac{\lambda \Delta t}{20}.%
\end{array}%
\end{equation*}%
To start the evolution of the vector of parameters $\delta ^{n},$ $\delta
^{0}$ must be calculated by using the periodic boundary condition and
initial condition $u(x,0).$ So, using the relations at the knots $%
u_{N}(x_{m},0)=u(x_{m},0)$, $m=0,1,2,...,N$ and $u_{N}^{^{\prime
}}(x_{0},0)=u^{^{\prime }}(x_{N},0)=0$ together with a variant of the Thomas
algorithm, the initial vector $\delta ^{0}$ is easily got from the following
matrix equation
\begin{equation*}
\left[
\begin{array}{cccccccc}
-3 & 0 & 3 &  &  &  &  &  \\
1 & 4 & 1 &  &  &  &  &  \\
&  &  & \ddots  &  &  &  &  \\
&  &  &  & 1 & 4 & 1 &  \\
&  &  &  & -3 & 0 & 3 &
\end{array}%
\right] \left[
\begin{array}{c}
\delta _{-}^{0}1 \\
\delta _{0}^{0} \\
\vdots  \\
\delta _{N}^{0} \\
\delta _{N+1}^{0}%
\end{array}%
\right] =\left[
\begin{array}{c}
u^{\prime }(x_{0},0) \\
u(x_{0},0) \\
\vdots  \\
u(x_{N},0) \\
u^{\prime }(x_{N},0)%
\end{array}%
\right] .
\end{equation*}

\section{ Stability analysis}

In this section, like other authors \cite{esen,roshan,khalifa,ras} our
stability analysis is based on the Von Neumann theory in which the growth
factor of a typical Fourier model defined as

\begin{equation}
\delta _{j}^{n}=g^{n}e^{ijkh},  \label{260}
\end{equation}%
where $k$ is mode number and $h$ is element greatness. To implement the
Fourier stability analysis, Eq. $(\ref{grlw})$ needs to be linearized by
assuming that the quantity $u^{p}$ in the nonlinear term $u^{p}u_{x}$ is
locally constant. Substituting the Fourier mode ($\ref{260})$ into
linearized scheme of ($\ref{15})$, we get
\begin{equation}
g=\frac{a-ib}{a+ib},  \label{16}
\end{equation}%
where
\begin{equation}
\begin{array}{l}
a=\left( 302+300\beta \right) \cos \left( \frac{\theta }{2}\right) h+\left(
57-270\beta \right) \cos \left( \frac{3\theta }{2}\right) h+\left( 1-30\beta
\right) \cos \left( \frac{5\theta }{2}\right) h, \\
b=120\lambda \Delta t\sin \left( \frac{\theta }{2}\right) h+75\lambda \Delta
t\sin \left( \frac{3\theta }{2}\right) h+3\lambda \Delta t\sin \left( \frac{%
5\theta }{2}\right) h,%
\end{array}
\label{17}
\end{equation}%
so that $|g|$ is $1$ and our linearized scheme is neutrally stable.

\section{Computer implementations and illustrations}

In this part, we introduce the results of the numerical experiments of our
algorithm for the solution of the GRLW Eq. $(\ref{intl})$--$(\ref{bndry})$
for a single solitary wave and an interaction of two solitary waves. We also
display the development of the Maxwellian initial condition into solitary
waves. In order to demonstrate how favorable our numerical algorithm
foresees the position and amplitude of the solution as the simulation
progresses, we provides for the \ following error norms:
\begin{equation*}
\mathit{L}_{2}=\left\Vert u^{exact}-u_{N}\right\Vert _{2}\simeq \sqrt{%
h\sum_{J=0}^{N}\left\vert u_{j}^{exact}-\left( u_{N}\right) _{j}\right\vert
^{2}},
\end{equation*}%
and
\begin{equation*}
\ \mathit{L}_{\infty }=\left\Vert u^{exact}-u_{N}\right\Vert _{\infty
}\simeq \max_{j}\left\vert u_{j}^{exact}-\left( u_{N}\right) _{j}\right\vert
.
\end{equation*}%
With the boundary condition $u\rightarrow 0$ for $x\rightarrow \pm \infty $
the exact solution of the GRLW equation is \cite{garcia}

\begin{equation*}
\mathit{u(x,t)}=\sqrt[p]{\frac{c(p+2)}{2p}\sec h^{2}[\frac{p}{2}\sqrt{\frac{c%
}{\mu (c+1)}}(x-(c+1)t-x_{0})]}
\end{equation*}%
where $\sqrt[p]{\frac{c(p+2)}{2p}}$ is amplitude, $c+1$ is the speed of the
wave traveling in the positive direction of the $x$-axis, $x_{0}$ is
arbitrary constant. There are three conserved quantities
\begin{align}
I_{1}& =\int_{-\infty }^{\infty }u(x,t)dx,  \notag \\
I_{2}& =\int_{-\infty }^{\infty }[{u^{2}(x,t)+\mu u_{x}^{2}(x,t)}]dx,~~~
\label{invrnt} \\
I_{3}& =\int_{-\infty }^{\infty }[{u^{4}(x,t)-\mu u^{2}(x,t)}]dx  \notag
\end{align}%
for the GRLW equation. These are correspond to mass, momentum and energy
respectively.

\subsection{Dispersion of a single solitary wave}

In our computational work for the first set, we prefer the parameters $p=2$,
$c=1$, $h=0.2$, $\Delta t$ $=0.025$, $\mu =1$, $x_{0}$ $=40$ with interval $%
[0,100]$ to match up with that of previous papers \cite%
{roshan,karakoc,karakoc1,gard3,khalifa}. These values yield the amplitude $%
1.0$ and the run of the algorithm is continued up to time $t=10$ over the
solution region. Analytical values of the invariants are $I_{1}$ $=4.442883$%
, $I_{2}$ $=3.299832$ and $I_{3}=1.414214.$ Values of the three invariants
as well as $\mathit{L}_{2}$ and $\ \mathit{L}_{\infty }$-error norms from
our method have been computed and tabulated in Table $(\ref{400})$.
Referring to Table $(\ref{400}),$ the error norms $\mathit{L}_{2}$ and $%
\mathit{L}_{\infty }$ remain less than $2.4154685\times 10^{-3}$ and $%
1.07968675\times 10^{-3},$ the invariants $I_{1},I_{2}$ and $I_{3}$ change
from their initial values by less than $3.10\times \ 10^{-4},$ $4.89\times \
10^{-4}$ and $4.79\times \ 10^{-4},$ respectively, throughout the
simulation. Also, our invariants are almost stable as time increases and the
agreement between numerical and analytic solutions is perfect. Hence our
method is acceptedly conservative. Comparisons with our results with exact
solution as well as the obtained values in \cite%
{roshan,karakoc,karakoc1,gard3,khalifa} have been made and listed in Table $(%
\ref{401})$ at $t=10$. This table evidentially indicates that the error
norms got by our method are marginally less than the others. The motion of
solitary wave using our scheme is plotted at time $t=0,5,10$ in Fig. $(\ref%
{500}).$\ It is obvious from the figure that the suggested method performs
the motion of propagation of a solitary wave admissibly, which moved to the
right with the preserved amplitude and shape. Initially, the amplitude of
solitary wave is $1.00000$ and its top position is pinpionted at $x=40$. At $%
t=10,$ its amplitude is recorded as $0.99928$ with center $x=60$. Thereby
the absolute difference in amplitudes over the time interval $[0,10]$ are
observed as $7.16\times 10^{-3}$. The quantile of error at discoint times
are depicted in Fig.$(\ref{501})$ . The error aberration varies from $%
-1\times 10^{-3}$ to $1\times 10^{-3}.$%
\begin{table}[h!]
\caption{Invariants and errors for single solitary wave with $%
p=2,c=1,h=0.2,\Delta t=0.025,\protect\mu =1,x\in \left[ 0,100\right] .$ }
\label{400}\vskip-1.cm
\par
\begin{center}
{\scriptsize
\begin{tabular}{cccccc}
\hline\hline
$Time$ & $I_{1}$ & $I_{2}$ & $I_{3}$ & $L_{2}\times 10^{3}$ & $L_{\infty
}\times 10^{3}$ \\ \hline
0 & 4.442866 & 3.299813 & 1.414214 & 0.000000 & 0.000000 \\
2 & 4.442940 & 3.299938 & 1.414330 & 1.948707 & 1.190456 \\
4 & 4.443005 & 3.300033 & 1.414425 & 2.362855 & 1.222540 \\
6 & 4.443068 & 3.300124 & 1.414515 & 2.449792 & 1.198936 \\
8 & 4.443129 & 3.300213 & 1.414604 & 2.448242 & 1.150862 \\
10 & 4.443175 & 3.300302 & 1.414692 & 2.415468 & 1.079686 \\ \hline\hline
\end{tabular}
}
\end{center}
\end{table}
\begin{table}[h!]
\caption{Comparisons of results for single solitary wave with $p=2,$ $c=1,$ $%
h=0.2,$ $\Delta t=0.025,$ $\protect\mu =1,$ $x\in \left[ 0,100\right] .$ }
\label{401}\vskip-1.cm
\par
\begin{center}
{\scriptsize
\begin{tabular}{cccccc}
\hline\hline
$Method$ & $I_{1}$ & $I_{2}$ & $I_{3}$ & $L_{2}\times 10^{3}$ & $L_{\infty
}\times 10^{3}$ \\ \hline
Analytic & 4.44288 & 3.29983 & 1.41421 & 0.000000 & 0.000000 \\
Our Method & 4.443175 & 3.300302 & 1.414692 & 2.415468 & 1.079686 \\
Petrov-Galerkin\cite{roshan} & 4.44288 & 3.29981 & 1.41416 & 3.00533 &
1.68749 \\
Septic Collocation First Scheme\cite{karakoc} & 4.442866 & 3.299822 &
1.414204 & 2.632463 & 1.393064 \\
Septic Collocation Second Scheme\cite{karakoc} & 4.442866 & 3.299715 &
1.414312 & 2.571481 & 1.340210 \\
Cubic Galerkin\cite{karakoc1} & 3.801670 & 2.888066 & 0.979294 & 13.291080 &
8.478107 \\
Cubic B-spline coll-CN\cite{gard3} & 4.442 & 3.299 & 1.413 & 16.39 & 9.24 \\
Cubic B-spline coll + PA-CN\cite{gard3} & 4.440 & 3.296 & 1.411 & 20.3 & 11.2
\\
Cubic B-spline collocation\cite{khalifa} & 4.44288 & 3.29983 & 1.41420 &
9.30196 & 5.43718 \\ \hline\hline
\end{tabular}
}
\end{center}
\end{table}
{\normalsize
\begin{figure}[h!]
{\normalsize {\scriptsize {\centering{\small %
\includegraphics[scale=.25]{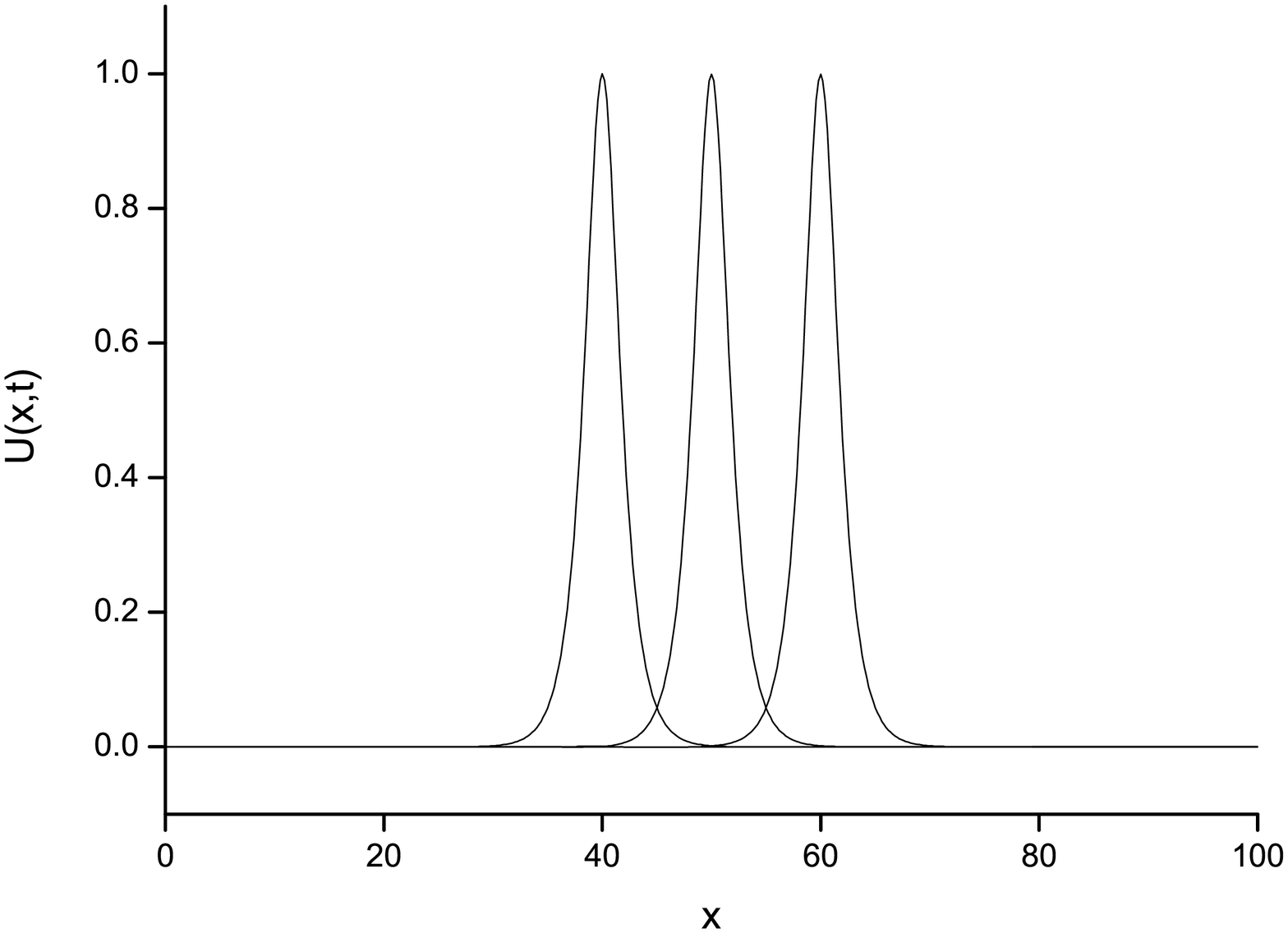}} } }  }
\caption{Motion of single solitary wave for $p=2$, $c=1$, $h=0.2$, $\Delta t$
$=0.025$ over the interval $[0,100]$ at $t=0,5,10.$}
\label{500}
\end{figure}
\begin{figure}[h!]
{\normalsize {\scriptsize {\centering{\small %
\includegraphics[scale=.25]{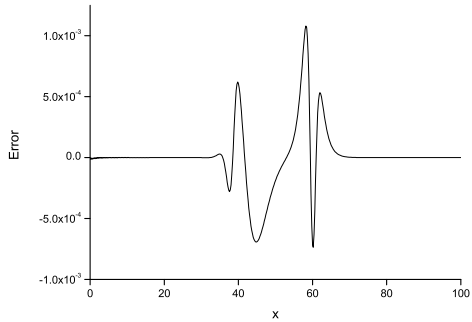}} } }  }
\caption{Error graph for $p=2$, $c=1$, $h=0.2$, $\Delta t$ $=0.025$ at $t=10$%
.}
\label{501}
\end{figure}
}

For the second set, we select the parameters $p=3,c=6/5,h=0.1,\Delta
t=0.025,\mu =1$, $x_{0}$ $=40$ with interval $[0,100]$ to coincide with that
of previous papers \cite{roshan,karakoc,karakoc1}. These parameters produce
the amplitude $1.0$ and the computations are carried out for times up to $%
t=10.$ The error norms $L_{2},$ $L_{\infty }$ and conservation quantities $%
I_{1},I_{2}$ and $I_{3}$ are computed, which are recorded in Table $(\ref%
{402})$. According to Table $(\ref{402})$ the error norms $\mathit{L}_{2}$
and $\mathit{L}_{\infty }$ remain less than $6.12802937\times 10^{-3}$ and $%
3.72213891\times 10^{-3},$ the invariants $I_{1},I_{2}$ and $I_{3}$ change
from their initial values by less than $9.75\times \ 10^{-5},$ $4.32\times \
10^{-5}$ and $4.78\times \ 10^{-4},$ respectively, during the simulation.
Also, our invariants are almost constant as time increases. Therefore our
method is satisfactorily conservative. In Table $(\ref{403})$ the
performance of the our new method is compared with other methods \cite%
{roshan,karakoc,karakoc1} at $t=10$. It is observed that errors of the
method \cite{roshan,karakoc,karakoc1} are considerably larger than those
obtained with the present scheme. Perspective views of the traveling
solitons are graphed at time $t=0,5,10$ in Fig.$(\ref{502}).$ It is clear
from the figure that the single soliton moved to the right with the
preserved amplitude and shape. The amplitude is $1.00000$ at $t=0$ and
located at $x=40$, while it is $0.99958$ at $t=10$ and located at $x=62$.
The absolute difference in amplitudes over the time interval $[0,10]$ are
found as $4.2\times 10^{-4}$. The aberration of error at discrete times are
modelled in Fig.$(\ref{503}).$ The error deviation varies from $-4\times
10^{-3}$ to $4\times 10^{-3}$.
\begin{table}[h!]
\caption{Invariants and errors for single solitary wave with $p=3,$ $c=6/5,$
$h=0.1,$ $\Delta t=0.025,$ $\protect\mu =1,$ $x\in \left[ 0,100\right] .$ }
\label{402}\vskip-1.cm
\par
\begin{center}
{\scriptsize
\begin{tabular}{cccccc}
\hline\hline
$Time$ & $I_{1}$ & $I_{2}$ & $I_{3}$ & $L_{2}\times 10^{3}$ & $L_{\infty
}\times 10^{3}$ \\ \hline
0 & 3.797185 & 2.881250 & 0.972968 & 0.000000 & 0.000000 \\
2 & 3.797187 & 2.881258 & 0.973414 & 1.700682 & 1.174285 \\
4 & 3.797187 & 2.881257 & 0.973473 & 2.805942 & 1.797229 \\
6 & 3.797187 & 2.881255 & 0.973486 & 3.899300 & 2.428864 \\
8 & 3.797200 & 2.881254 & 0.973487 & 5.007404 & 3.073644 \\
10 & 3.797282 & 2.881293 & 0.973446 & 6.128029 & 3.722138 \\ \hline\hline
\end{tabular}
}
\end{center}
\end{table}
\begin{table}[h!]
\caption{Comparisons of results for single solitary wave with $p=3,$ $c=6/5,$
$h=0.1,$ $\Delta t=0.025,$ $\protect\mu =1,$ $x\in \left[ 0,100\right] .$ }
\label{403}\vskip-1.cm
\par
\begin{center}
{\scriptsize
\begin{tabular}{cccccc}
\hline\hline
$Method$ & $I_{1}$ & $I_{2}$ & $I_{3}$ & $L_{2}\times 10^{3}$ & $L_{\infty
}\times 10^{3}$ \\ \hline
Our Method & 3.797282 & 2.881293 & 0..973446 & 6.128029 & 3.722138 \\
Petrov-Galerkin\cite{roshan} & 3.79713 & 2.88123 & 0.972243 & 7.76745 &
4.70875 \\
Septic Collocation First Scheme\cite{karakoc} & 3.797185 & 2.881252 &
0.973145 & 8.972983 & 5.175982 \\
Septic Collocation Second Scheme\cite{karakoc} & 3.797133 & 2.881089 &
0.973128 & 7.778169 & 4.441873 \\
Cubic Galerkin\cite{karakoc1} & 3.801670 & 2.888066 & 0.979294 & 13.291080 &
8.478107 \\ \hline\hline
\end{tabular}
}
\end{center}
\end{table}
{\normalsize
\begin{figure}[h!]
{\normalsize {\scriptsize {\centering{\small %
\includegraphics[scale=.25]{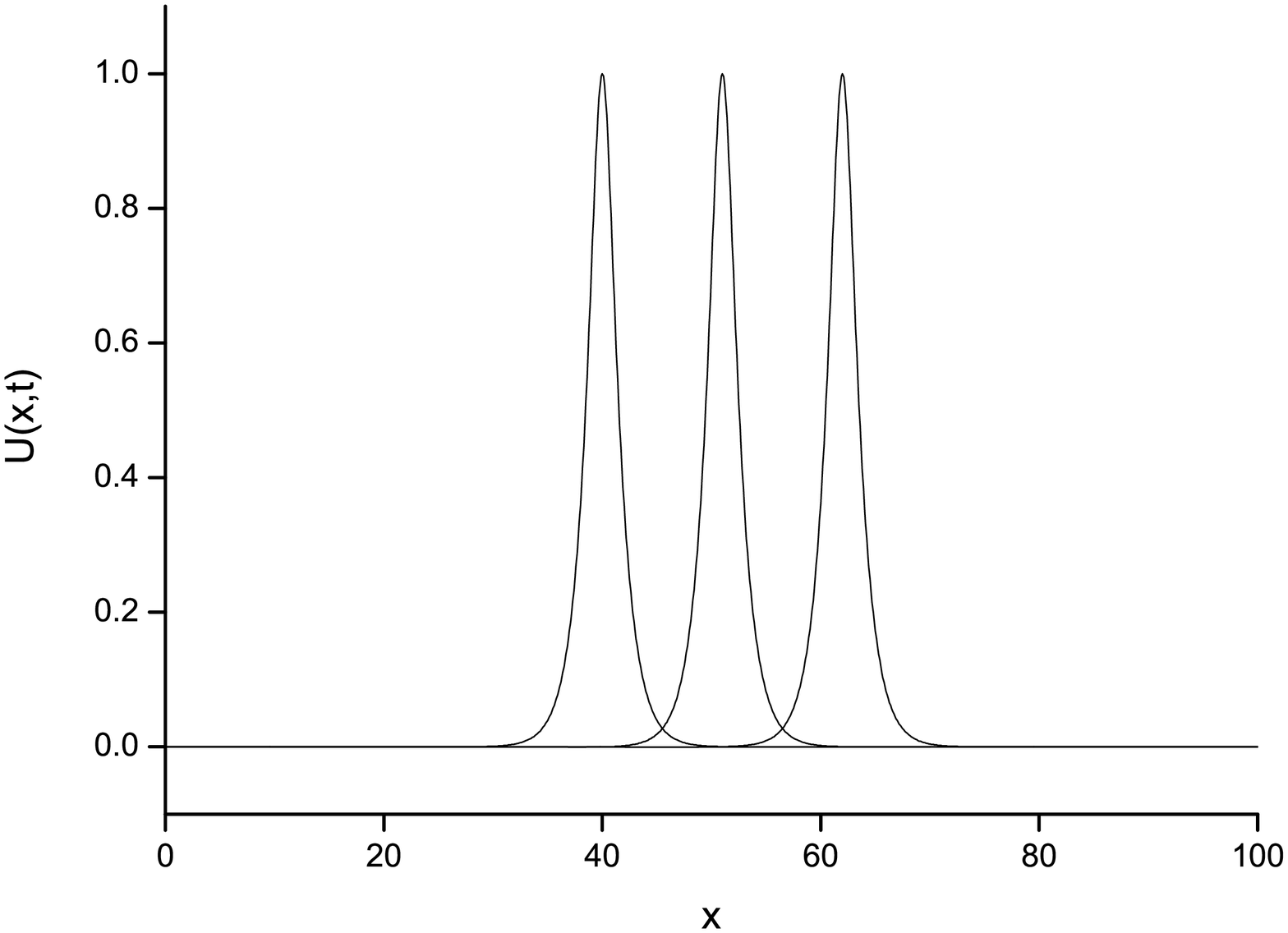}} } }  }
\caption{Motion of single solitary wave for $p=3$, $c=6/5$, $h=0.1$, $\Delta
t$ $=0.025$ over the interval $[0,100]$ at $t=0,5,10.$}
\label{502}
\end{figure}
\begin{figure}[h!]
{\normalsize {\scriptsize {\centering{\small %
\includegraphics[scale=.25]{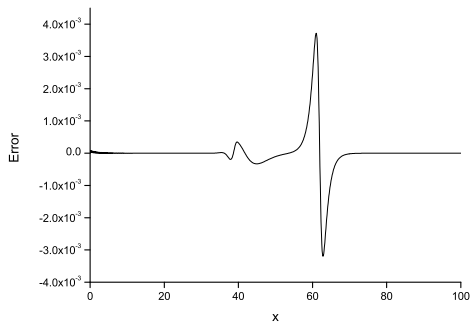}} } }  }
\caption{Error graph for $p=3$, $c=6/5$, $h=0.1$, $\Delta t$ $=0.025$ at $%
t=10$.}
\label{503}
\end{figure}
\ \ \ \ }

Finally, we choose the parameters $p=4,c=4/3,h=0.1,\Delta t=0.01,\mu =1$, $%
x_{0}$ $=40$ over the region $[0,100]$ to compare with those of earlier
papers \cite{roshan,karakoc,karakoc1}. These parameters lead to amplitude $%
1.0$ and the simulations are executed to time $t=10$ to invent the error
norms $L_{2}$ and $L_{\infty }$ and the numerical invariants $I_{1},I_{2}$
and $I_{3}.$ For these values of the parameters, the conservation properties
and the $L_{2}$-error as well as the $L_{\infty }$-error norms have been
given in Table$(\ref{4040})$ for various values of the time level $t$. It
can be noted from Table $(\ref{4040}),$ the error norms $\mathit{L}_{2}$ and
$\mathit{L}_{\infty }$ remain less than $1.28342020\times 10^{-3}$ and $%
0.82165081\times 10^{-3},$ the invariants $I_{1},I_{2}$ and $I_{3}$ change
from their initial values by less than $9.02\times \ 10^{-5},$ $5.08\times \
10^{-5}$ and $7.97\times \ 10^{-4},$ respectively, throughout the
simulation. Also, our invariants are almost constant as time increases.
Therefore we can say our method is sensibly conservative. The comparison
between the results obtained by the present method with those in the other
studies \cite{roshan,karakoc,karakoc1} is also documented in Table $(\ref%
{4050})$. It is noticeably seen from the table that errors of the present
method are radically less than those obtained with the earlier schemes \cite%
{roshan,karakoc,karakoc1}. For visual representation, the simulations of
single soliton for values $p=4,c=4/3,h=0.1,\Delta t=0.01$ at times $t=0,5$
and $10$ are illustrated in Figure $(\ref{504})$. It is understood from this
figure that the numerical scheme performs the motion of propagation of a
single solitary wave, which moves to the right at nearly unchanged speed and
conserves its amplitude and shape with increasing time. The amplitude is $%
1.00000$ at $t=0$ and located at $x=40$, while it is $0.99892$ at $t=10$ and
located at $x=63.3$. The absolute difference in amplitudes at times $t=0$
and $t=10$ is $1.08\times 10^{-3}$ so that there is a little change between
amplitudes. Error distributions at time $t=10$ are shown graphically in
Figure $(\ref{505})$. As it is seen, the maximum errors are between $%
-6\times 10^{-4}$ to $1\times 10^{-3}$ and occur around the central position
of the solitary wave.
\begin{table}[h!]
\caption{Invariants and errors for single solitary wave with $p=4,$ $c=4/3,$
$h=0.1,$ $\Delta t=0.01,$ $\protect\mu =1,$ $x\in \left[ 0,100\right] .$ }
\label{4040}\vskip-1.cm
\par
\begin{center}
{\scriptsize
\begin{tabular}{cccccc}
\hline\hline
$Time$ & $I_{1}$ & $I_{2}$ & $I_{3}$ & $L_{2}\times 10^{3}$ & $L_{\infty
}\times 10^{3}$ \\ \hline
0 & 3.468709 & 2.671691 & 0.729204 & 0.000000 & 0.000000 \\
2 & 3.468718 & 2.671714 & 0..729969 & 0.967786 & 0.708600 \\
4 & 3.468719 & 2.671714 & 0.730017 & 1.040242 & 0.591250 \\
6 & 3.468720 & 2.671714 & 0.730027 & 1.102854 & 0.611363 \\
8 & 3.468731 & 2.671714 & 0.730028 & 1.183442 & 0.715175 \\
10 & 3.468799 & 2.671742 & 0.730001 & 1.283420 & 0.821650 \\ \hline\hline
\end{tabular}
}
\end{center}
\end{table}
\begin{table}[h!]
\caption{Comparisons of results for single solitary wave with $p=4,$ $c=4/3,$
$h=0.1,$ $\Delta t=0.01,$ $\protect\mu =1,$ $x\in \left[ 0,100\right] .$ }
\label{4050}\vskip-1.cm
\par
\begin{center}
{\scriptsize
\begin{tabular}{cccccc}
\hline\hline
$Method$ & $I_{1}$ & $I_{2}$ & $I_{3}$ & $L_{2}\times 10^{3}$ & $L_{\infty
}\times 10^{3}$ \\ \hline
Our Method & 3.468799 & 2.671742 & 0..730001 & 1.283420 & 0.821650 \\
Petrov-Galerkin\cite{roshan} & 3.46866 & 2.67168 & 0.728881 & 2.46065 &
1.56620 \\
Septic Collocation First Scheme\cite{karakoc} & 3.468709 & 2.671696 &
0.729258 & 3.351740 & 2.049733 \\
Septic Collocation Second Scheme\cite{karakoc} & 3.468671 & 2.671658 &
0.729237 & 2.698709 & 1.656002 \\
Cubic Galerkin\cite{karakoc1} & 3.470439 & 2.674445 & 0.731987 & 1.511394 &
0.857585 \\ \hline\hline
\end{tabular}
}
\end{center}
\end{table}
{\normalsize
\begin{figure}[h!]
{\normalsize {\scriptsize {\centering{\small %
\includegraphics[scale=.25]{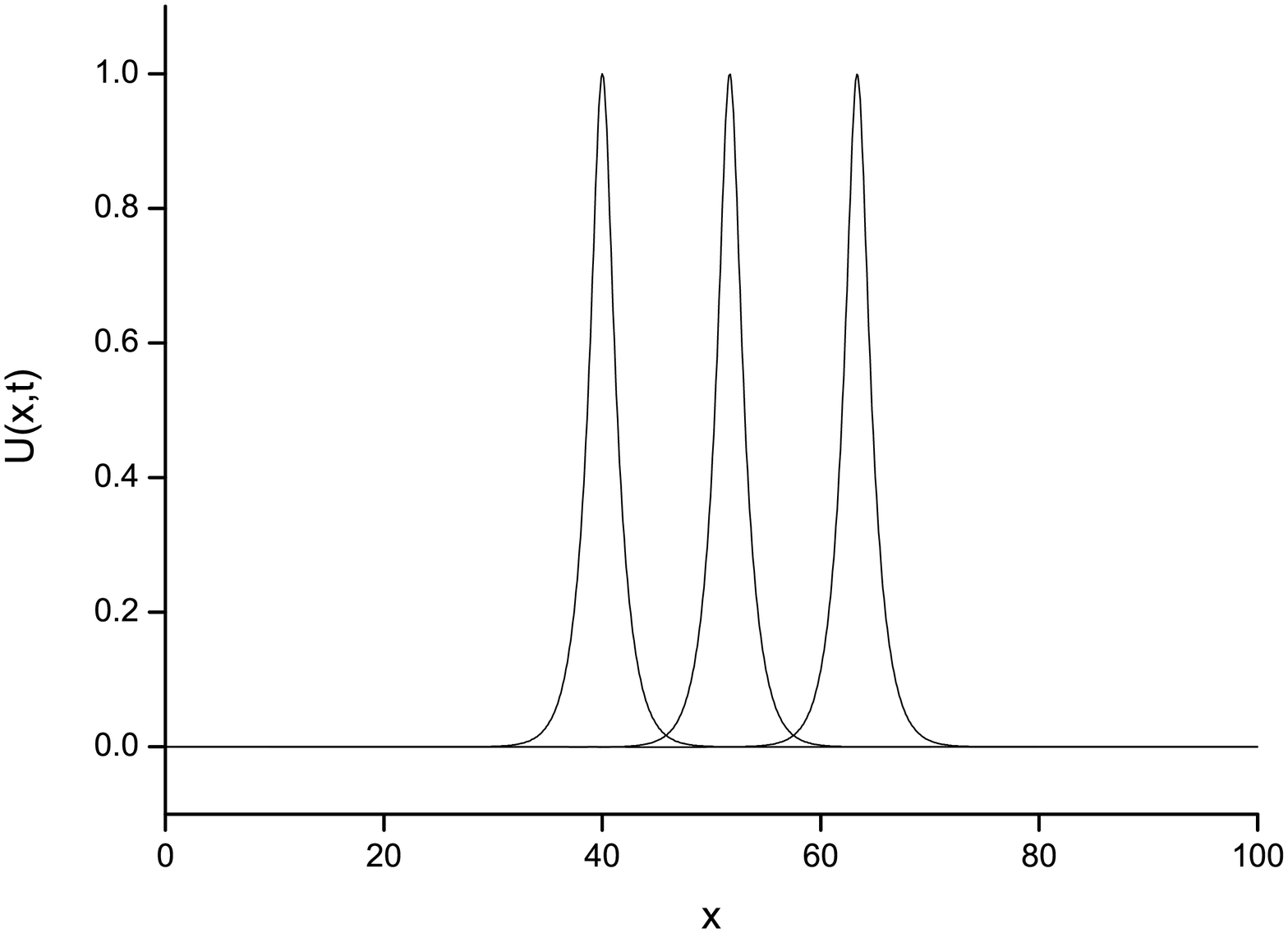}} } }  }
\caption{Motion of single solitary wave for $p=4$, $c=4/3$, $h=0.1$, $\Delta
t$ $=0.01$ over the interval $[0,100]$ at $t=0,5,10.$}
\label{504}
\end{figure}
\begin{figure}[h!]
{\normalsize {\scriptsize {\centering{\small %
\includegraphics[scale=.25]{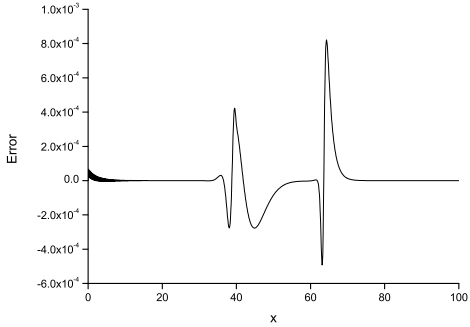}} } }  }
\caption{Error graph for $p=4$, $c=4/3$, $h=0.1$, $\Delta t$ $=0.01$ at $%
t=10 $.}
\label{505}
\end{figure}
}

\subsection{Interaction of two solitary waves}

As a second problem, we have focused on the behavior of the interaction of
two solitary waves having different amplitudes and moving in the same
direction. We provide for the GRLW equation with initial \ conditions given
by the linear sum of two well separated solitary waves of various amplitudes

\begin{equation}
u(x,0)=\sum_{j=1}^{2}\sqrt[p]{\frac{c_{j}(p+2)}{2p}\sec h^{2}[\frac{p}{2}%
\sqrt{\frac{c_{j}}{\mu (c_{j}+1)}}(x-x_{j})]},  \label{ini. for two solit}
\end{equation}%
where $c_{j}$ and $x_{j}$, $j=1,2$ are arbitrary constants. For the
simulation, we firstly choose $p=3,c_{1}=48/5,c_{2}=6/5,$ $h=0.1$, $\Delta
t=0.01,\mu =1$ over the interval $0\leq x\leq 120.$ The amplitudes are in
the ratio $2:1$. Calculations are performed to time $t=6$. The results are
listed in Table $(\ref{4051})$. Referring to this table, it is noticed that
the numerical values of the invariants are very closed with the methods \cite%
{roshan,karakoc,karakoc1} during the computer run. The initial function was
placed with the larger wave to the left of the smaller one as seen in the
Fig. $(\ref{40510})$a. Both waves move to the right with velocities
dependent upon their magnitudes. According to Fig. $(\ref{40510})$, the
larger wave catches up with the smaller wave at about $t=3$, the overlapping
process continues until $t=4$, then two solitary waves emerge from the
interaction and resume their former shapes and amplitudes. At $t=6$, the
magnitude of the smaller wave is $1.00029$ on reaching position $x=60.0$,
and of the larger wave $1.99213$ having the position $x=85.3$, so that the
difference in amplitudes is $0.00029$ for the smaller wave and $0.00787$ for
the larger wave. The changes of the invariants for this case are
satisfactorily small. Secondly, to ensure an interaction of two solitary
waves take place, calculation is carried out with the parameters \ $%
p=4,c_{1}=64/3,c_{2}=4/3,$ $h=0.125$, $\Delta t=0.01,\mu =1$ over the
interval $0\leq x\leq 200.$ The parameters give solitary waves of different
amplitudes $2$ and $1$ having centers at $x_{1}=20$ and $x_{2}=80.$ The
results are given in Table $(\ref{40530}).$ According to the this table, it
is realized that the numerical values of the invariants are very closed with
the methods \cite{roshan,karakoc,karakoc1} during the computer run. The
initial function was placed with the larger wave to the left of the smaller
one as seen in the Fig. $(\ref{4052})$a. Both waves move to the right with
velocities dependent upon their magnitudes. According to Fig. $(\ref{4052})$%
, the larger wave catches up with the smaller wave at about $t=3$, the
overlapping process continues until $t=5$, then two solitary waves emerge
from the interaction and resume their former shapes and amplitudes.
\begin{table}[h]
\caption{Invariants for interaction of two solitary waves with $p=3.$ }
\label{4051}\vskip-1.cm
\par
\begin{center}
{\scriptsize
\begin{tabular}{cccccc}
\hline\hline
&  &  &  &  &  \\ \hline
& $t$ & $0$ & $2$ & $4$ & $6$ \\ \hline
& Our Method & 9.690777 & 9.690777 & 9.690777 & 9.690777 \\
& \cite{roshan} & 9.69075 & 9.69074 & 9.69074 & 9.69074 \\
$I_{1}$ & \cite{karakoc} First & 9.690777 & 9.690777 & 9.690777 & 9.690778
\\
& \cite{karakoc} Second & 9.690777 & 9.688117 & 9.686015 & 9.683462 \\
& \cite{karakoc1} & 9.6907 & 9.6906 & 9.6898 & 9.6901 \\
& Our Method & 12.944360 & 12.928161 & 12.957476 & 12.988509 \\
& \cite{roshan} & 12.9444 & 12.9452 & 12.9453 & 12.9454 \\
$I_{2}$ & \cite{karakoc} First & 12.944391 & 12.944392 & 12.944393 &
12.944394 \\
& \cite{karakoc} Second & 12.944391 & 12.939062 & 12.970312 & 13.002753 \\
& \cite{karakoc1} & 12.9443 & 12.9440 & 12.9418 & 12.9426 \\
& Our Method & 17.018706 & 17.034905 & 17.005590 & 16.974557 \\
& \cite{roshan} & 17.0184 & 16.9835 & 16.9261 & 16.9113 \\
$I_{3}$ & \cite{karakoc} First & 17.018675 & 17.02567 & 16.981696 & 16.952024
\\
& \cite{karakoc} Second & 17.018675 & 17.02400 & 16.992754 & 16.960313 \\
& \cite{karakoc1} & 17.0187 & 17.0324 & 16.9849 & 16.9557 \\ \hline\hline
\end{tabular}
}
\end{center}
\end{table}
\begin{table}[h]
\caption{Invariants for interaction of two solitary waves with $p=4.$ }
\label{40530}\vskip-1.cm
\par
\begin{center}
{\scriptsize
\begin{tabular}{cccccc}
\hline\hline
&  &  &  &  &  \\ \hline
& $t$ & $0$ & $2$ & $4$ & $6$ \\ \hline
& Our Method & 8.834272 & 8.834272 & 8.834272 & 8.834272 \\
& \cite{roshan} & 8.83427 & 8.84204 & 8.84209 & 8.83434 \\
$I_{1}$ & \cite{karakoc} First & 8.834272 & 8.834160 & 8.834053 & 8.8339467
\\
& \cite{karakoc} Second & 8.834272 & 8.564186 & 8.435464 & 8.327161 \\
& \cite{karakoc1} & 8.8342 & 8.7089 & 8.6518 & 8.6134 \\
& Our Method & 12.170706 & 11.339311 & 11.209384 & 15.812521 \\
& \cite{roshan} & 12.1697 & 12.3700 & 12.5703 & 12.6103 \\
$I_{2}$ & \cite{karakoc} First & 12.170887 & 12.170537 & 12.170205 &
12.169873 \\
& \cite{karakoc} Second & 12.170887 & 11.939598 & 11.977097 & 11.814722 \\
& \cite{karakoc1} & 12.1707 & 11.7871 & 11.6179 & 11.4992 \\
& Our Method & 14.029604 & 14.860999 & 14.990927 & 10.387789 \\
& \cite{roshan} & 14.0302 & 13.9607 & 13.9805 & 14.6974 \\
$I_{3}$ & \cite{karakoc} First & 14.029423 & 14.413442 & 14.351624 &
14.292901 \\
& \cite{karakoc} Second & 14.029423 & 14.260712 & 14.223214 & 14.385588 \\
& \cite{karakoc1} & 14.0296 & 12.9204 & 12.1972 & 11.9640 \\ \hline\hline
\end{tabular}
}
\end{center}
\end{table}
\begin{figure}[h]
\centering{\small \includegraphics[scale=.20]{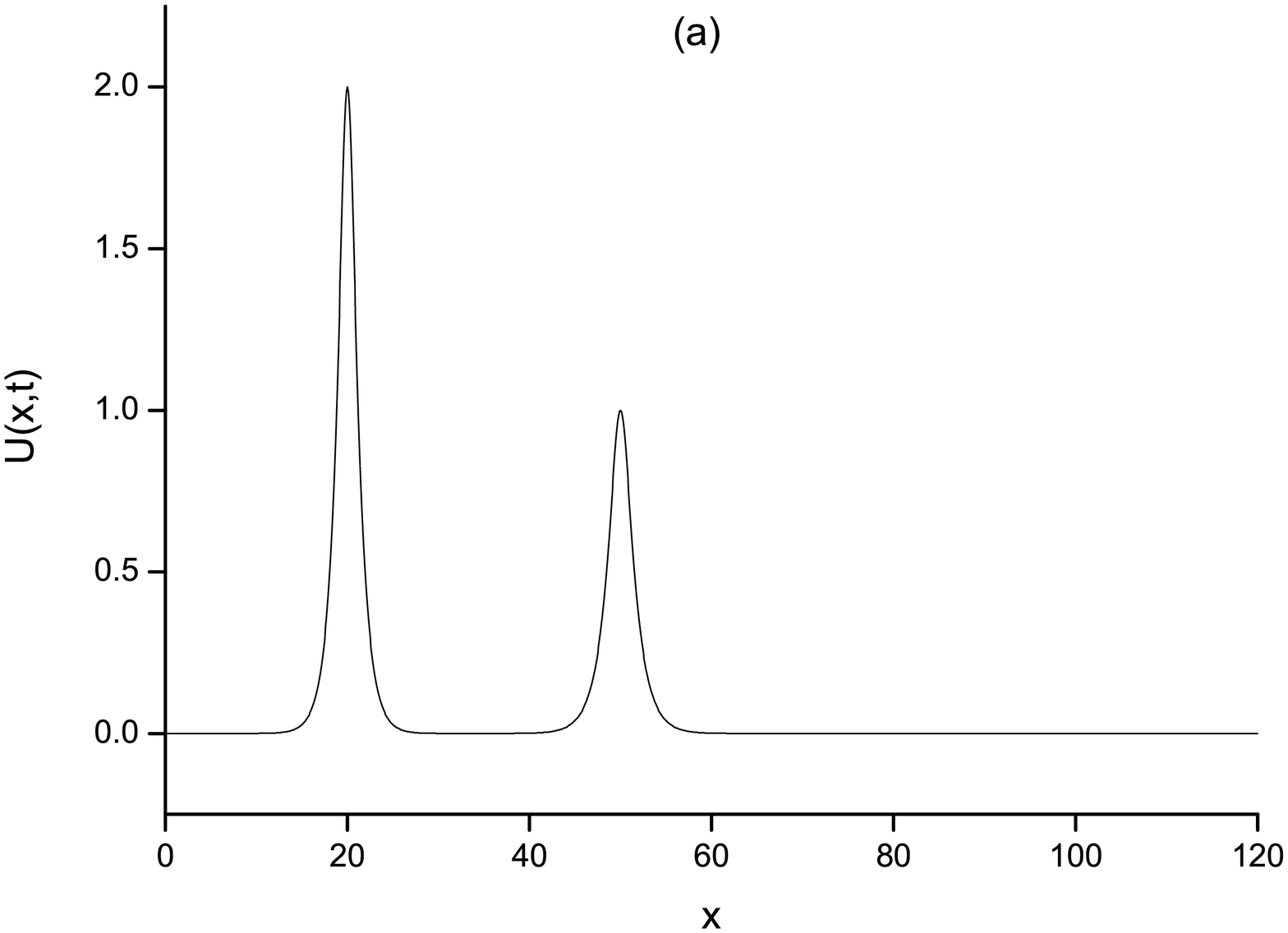} %
\includegraphics[scale=.20]{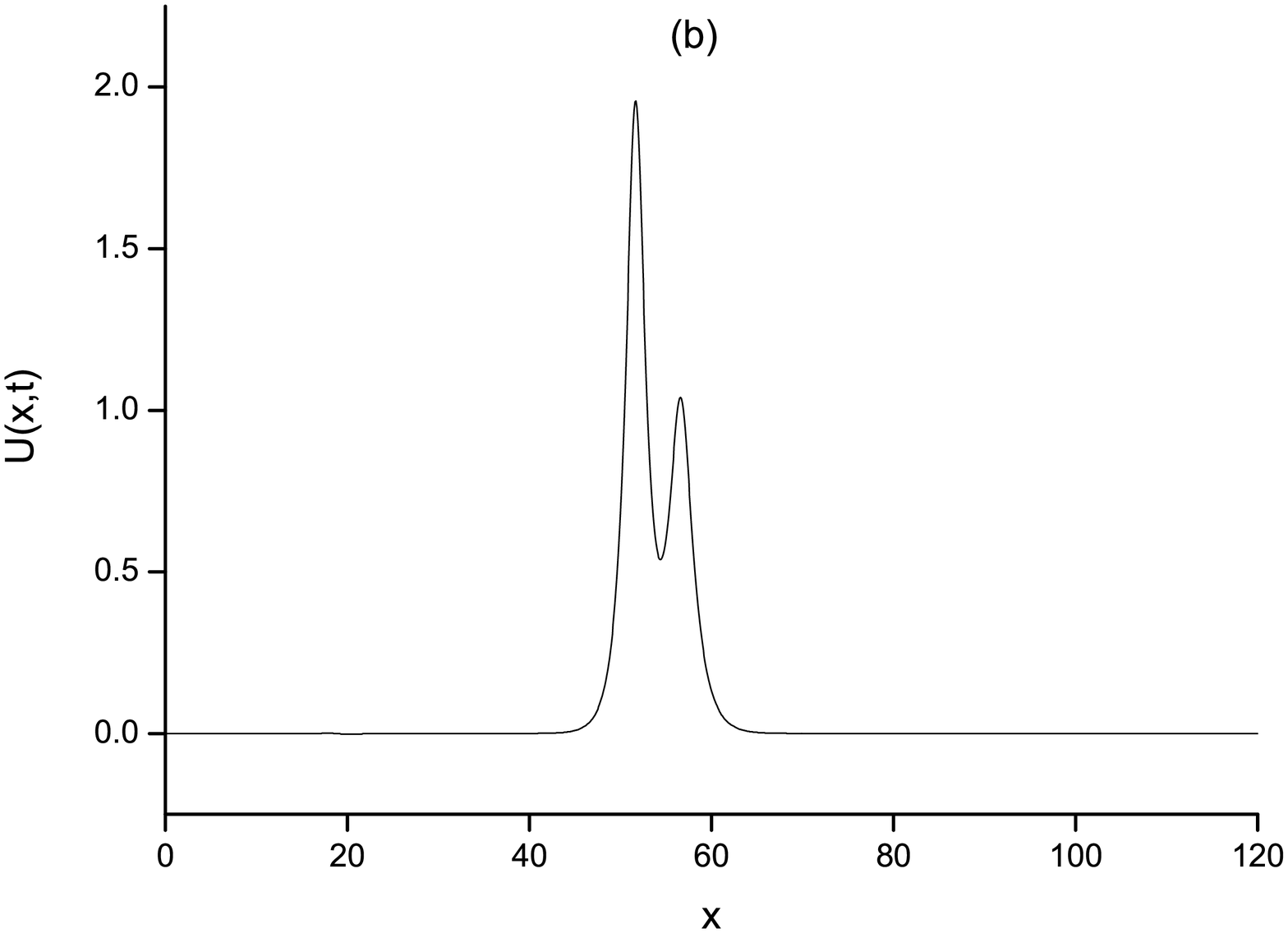}}
\par
{\small \includegraphics[scale=.20]{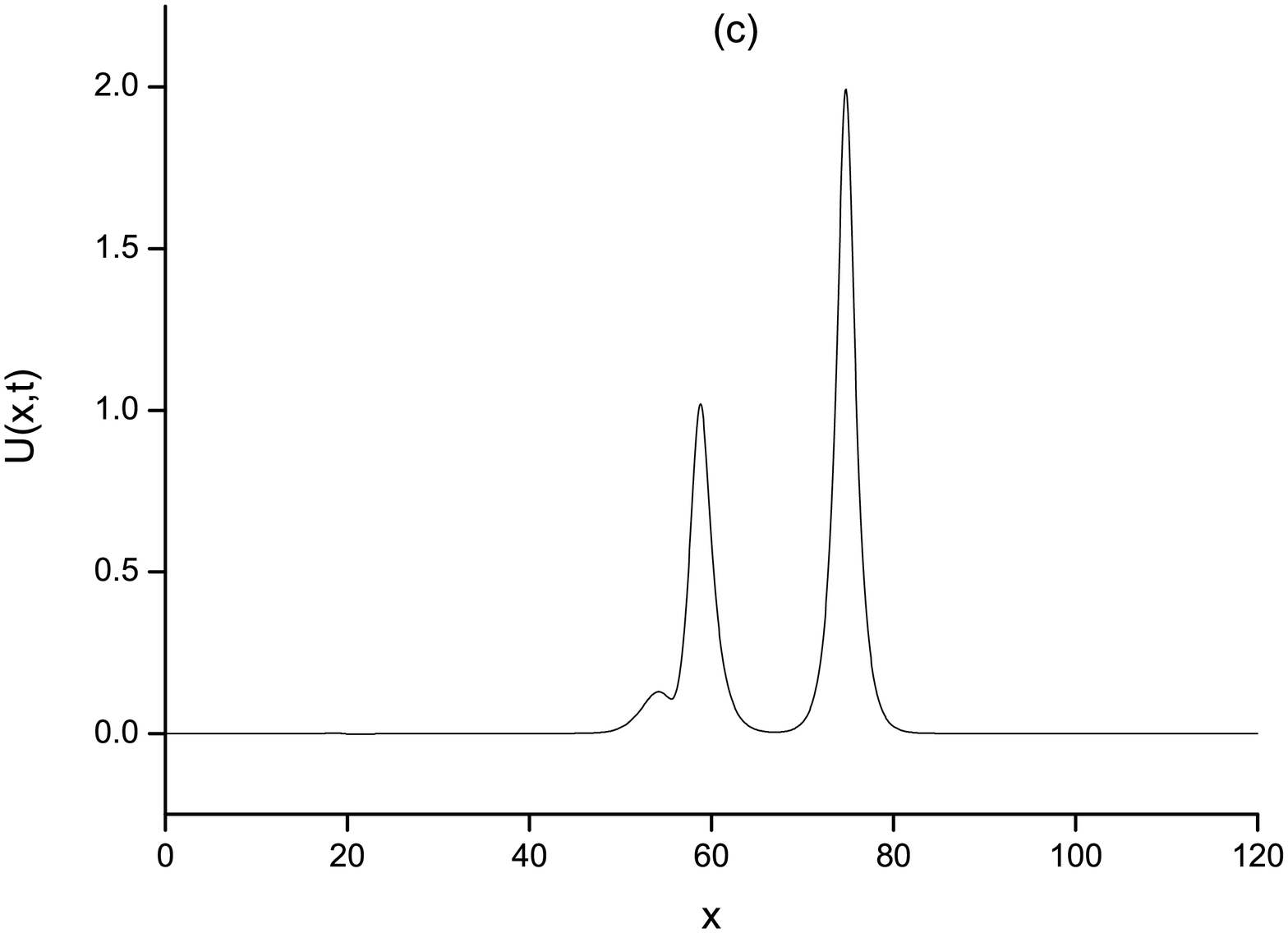} %
\includegraphics[scale=.20]{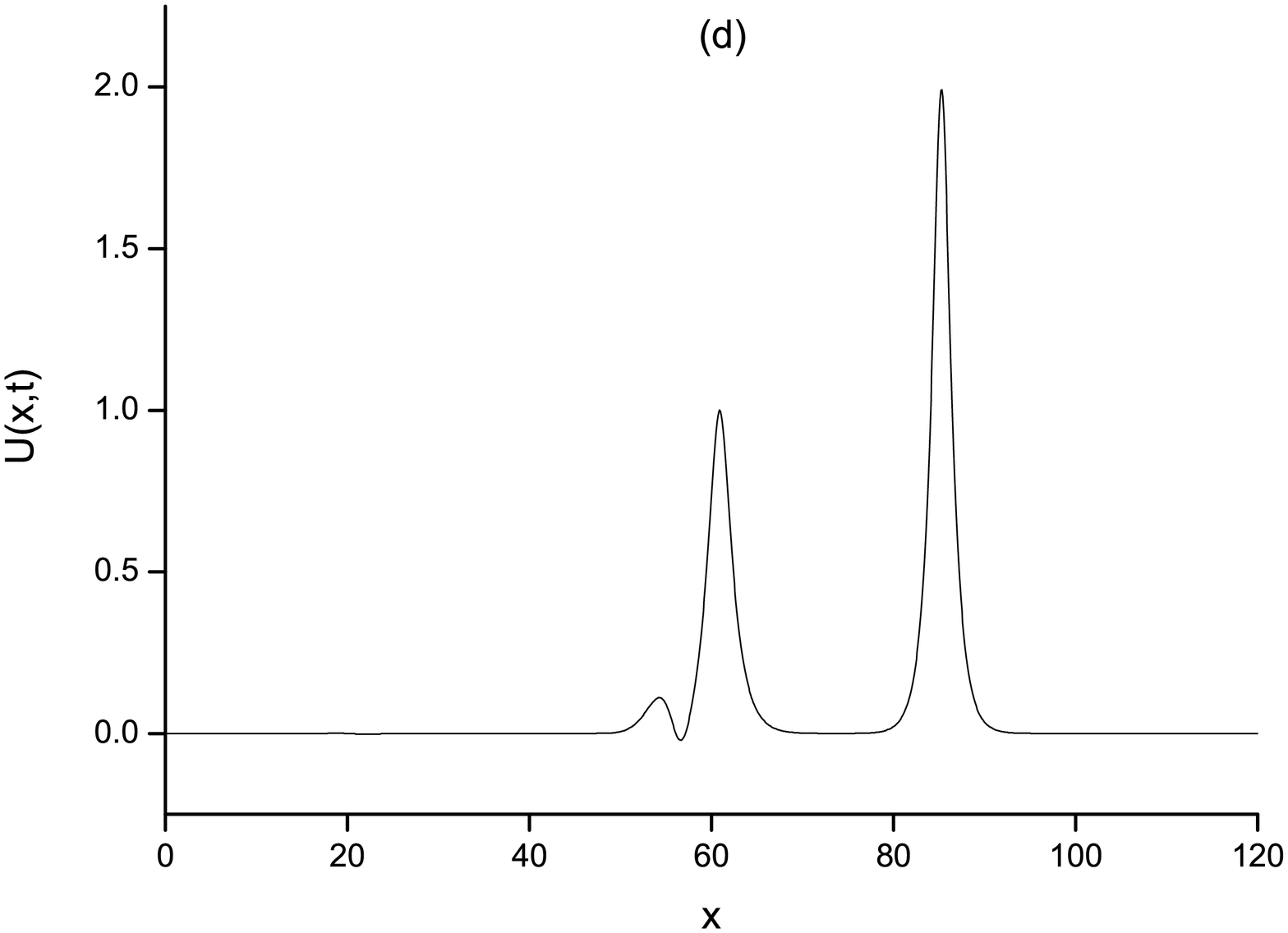}}
\caption{Interaction of two solitary waves at $p=3;$ $(a)t=0,$ $(b)t=3,$ $%
(c)t=5,$ $(d)t=6.$}
\label{40510}
\end{figure}
\begin{figure}[h]
\centering{\small \includegraphics[scale=.20]{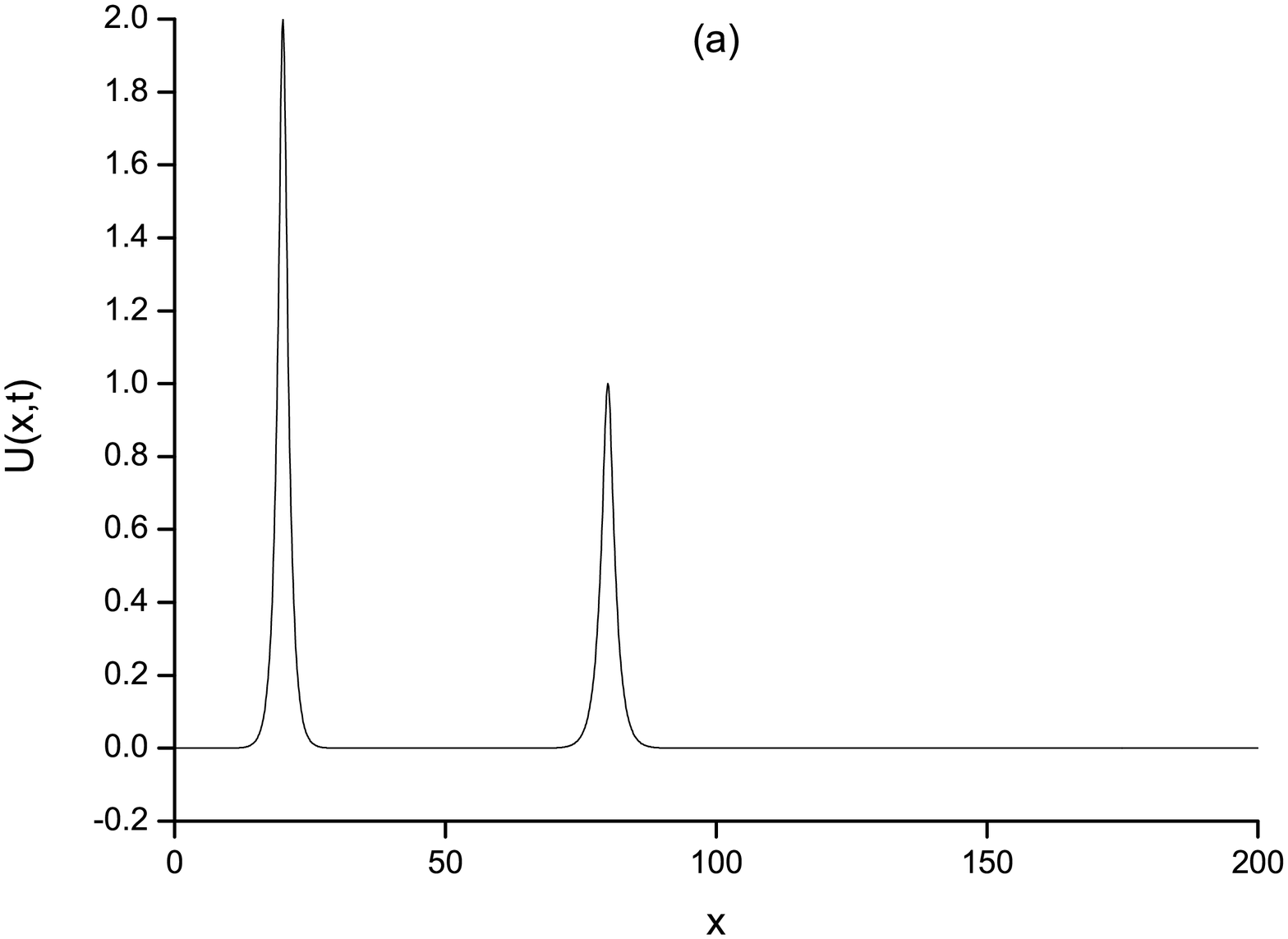} %
\includegraphics[scale=.20]{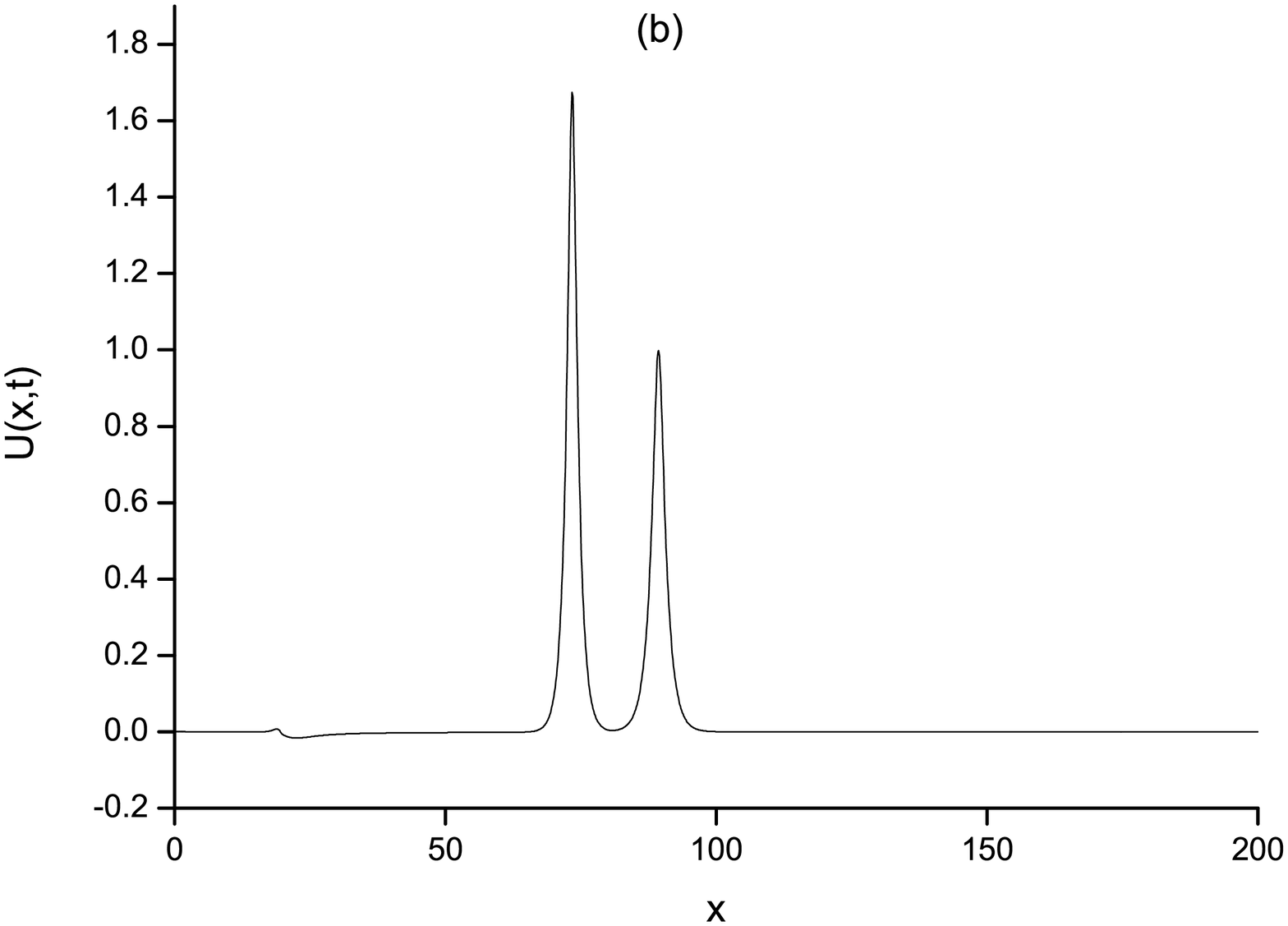}}
\par
{\small \includegraphics[scale=.20]{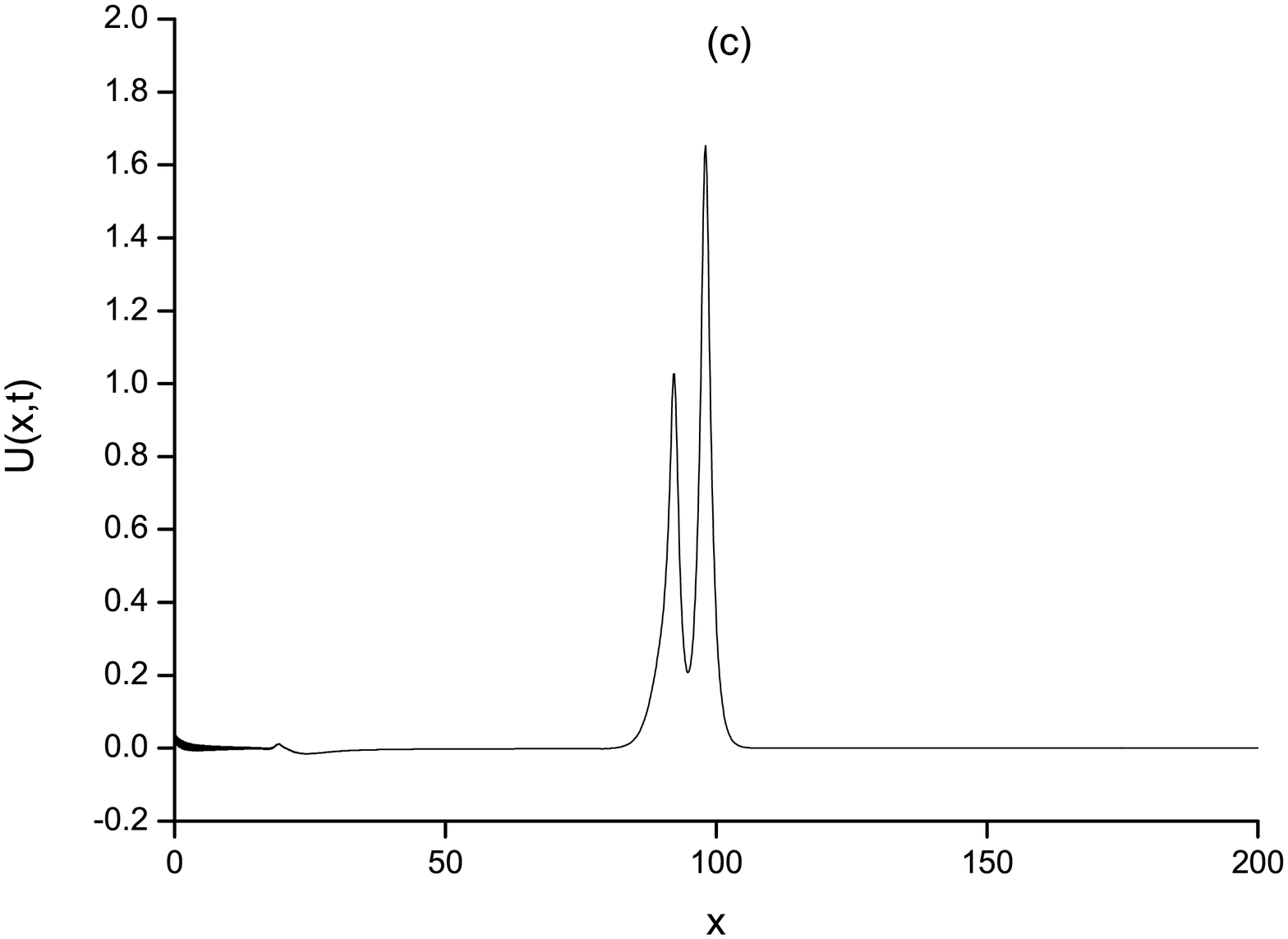} %
\includegraphics[scale=.20]{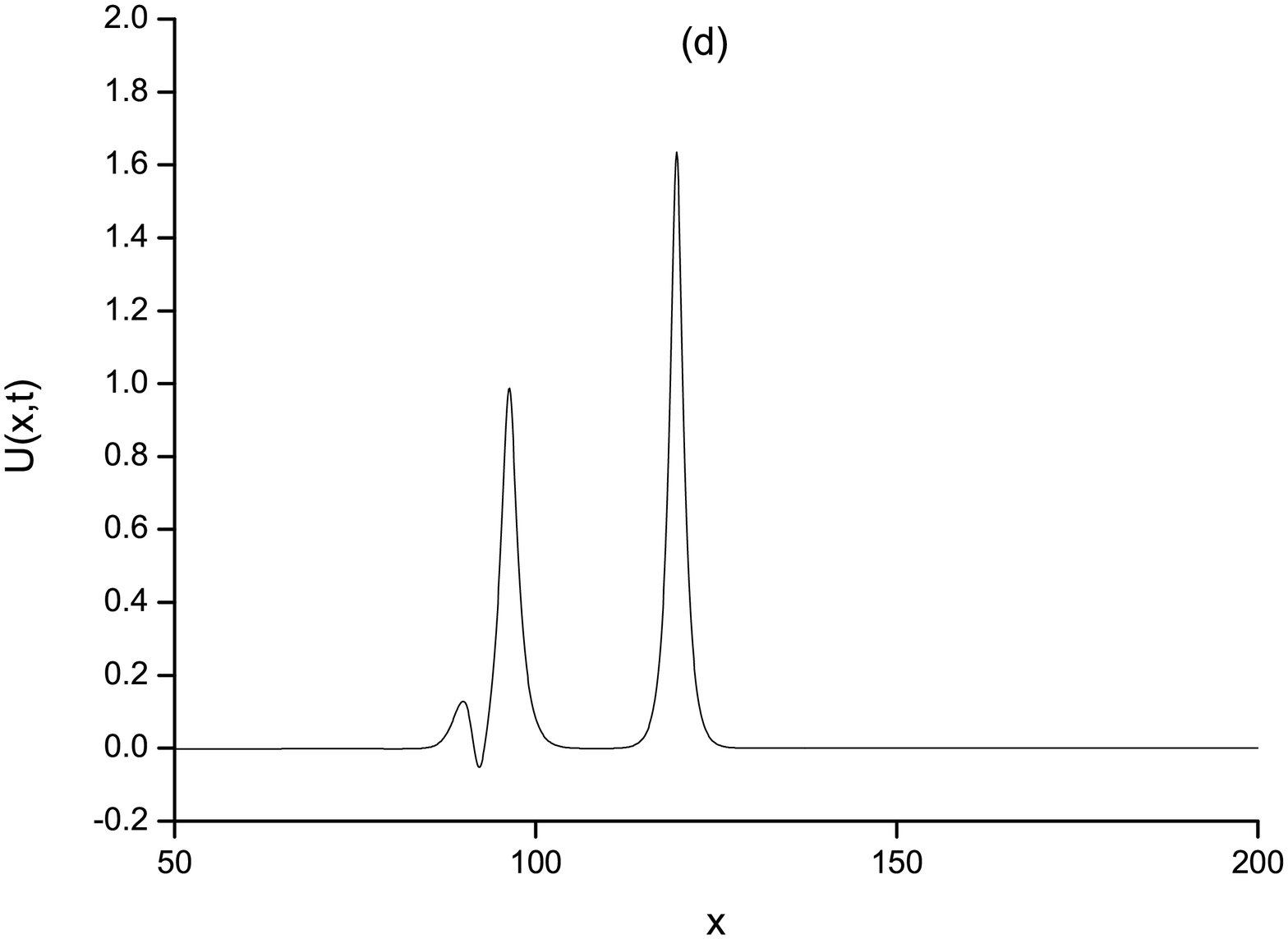}}
\caption{Interaction of two solitary waves at $p=4;$ $(a)t=0,$ $(b)t=2,$ $%
(c)t=4,$ $(d)t=6.$}
\label{4052}
\end{figure}

\subsection{The Maxwellian initial condition}

Finally, we have examined the evolution of an initial Maxwellian pulse into
solitary waves, arising as initial condition

\begin{equation}
u(x,0)=\exp (-(x-40)^{2}).
\end{equation}%
For this problem, the behavior of the solution depends on the value of $%
\mu
$ \cite{kaya2,roshan}. Therefore, we chose the values of $%
\mu
=0.1,$ $%
\mu
=0.05$ and $%
\mu
=0.025$ for $p=2,3,4$. The numerical computations are done up to $t=0.05$.
Calculated numerical invariants at different values of $t$ are shown in
Table $(\ref{4053})$ and it is seen that calculated invariant values are
satisfactorily constant. For $p=2$ and $%
\mu
=0.1$; the variation of invariants $I_{1},I_{2}$ and $I_{3}$ from initial
variants changes less than $1.02\times 10^{-3},4.48\times 10^{-3}$ and $%
8.69\times 10^{-3}$percent, respectively, and for $\mu =0.05;$ $2.01\times
10^{-3},8.35\times 10^{-3}$ and $17.92\times 10^{-3}$percent, respectively,
and for $\mu =0.025;$ $3.19\times 10^{-3},12.72\times 10^{-3}$ and $%
29.28\times 10^{-3}$percent, respectively, for $p=3$ and $%
\mu
=0.1$; the variation of invariants $I_{1},I_{2}$ and $I_{3}$ from initial
variants changes less than $8.60\times 10^{-3},2.77\times 10^{-2}$ and $%
4.75\times 10^{-1}$percent, respectively, and for $\mu =0.05;$ $17.45\times
10^{-3},54.32\times 10^{-3}$ and $6.62\times 10^{-1}$percent, respectively,
and for $\mu =0.025;$ $30.92\times 10^{-3},93.27\times 10^{-3}$ and $%
7.72\times 10^{-1}$percent, respectively, for $p=4$ and $%
\mu
=0.1$; the variation of invariants $I_{1},I_{2}$ and $I_{3}$ from initial
variants changes less than $38.82\times 10^{-3},1.09\times 10^{-3}$ and $1.76
$ percent, respectively, and for $\mu =0.05;$ $67.23\times
10^{-3},1.99\times 10^{-1}$ and $2.48$ percent, respectively, and for $\mu
=0.025;$ $1.04\times 10^{-1},3.31\times 10^{-1}$ and $3.0$ percent,
respectively. The development of the Maxwellian initial condition into
solitary waves is shown in Fig. $(\ref{4055})$.
\begin{table}[h]
\caption{Maxwellian initial condition for different values of $\protect\mu .$%
}
\label{4053}\vskip-1.cm
\par
\begin{center}
{\scriptsize
\begin{tabular}{ccccccccccc}
\hline\hline
$\mu $ & $t$ & \multicolumn{3}{c}{p=2} & \multicolumn{3}{c}{p=3} &
\multicolumn{3}{c}{p=4} \\ \hline
&  & $I_{1}$ & $I_{2}$ & $I_{3}$ & $I_{1}$ & $I_{2}$ & $I_{3}$ & $I_{1}$ & $%
I_{2}$ & $I_{3}$ \\ \hline
& 0.01 & 1.772481 & 1.378659 & 0.760911 & 1.772481 & 1.378655 & 0.760779 &
1.772422 & 1.378551 & 0.760310 \\
0.1 & 0.03 & 1.772475 & 1.378639 & 0.760890 & 1.772435 & 1.378538 & 0.759621
& 1.772189 & 1.378049 & 0.755954 \\
& 0.05 & 1.772463 & 1.378599 & 0.760847 & 1.772328 & 1.378278 & 0.757292 &
1.771793 & 1.377149 & 0.747477 \\
& 0.01 & 1.772480 & 1.315994 & 0.823572 & 1.772480 & 1.315988 & 0.823384 &
1.772382 & 1.315812 & 0.822659 \\
0.05 & 0.03 & 1.772469 & 1.315959 & 0.823526 & 1.772388 & 1.315772 & 0.821657
& 1.771994 & 1.314976 & 0.816132 \\
& 0.05 & 1.772445 & 1.315887 & 0.823429 & 1.772171 & 1.315282 & 0.818116 &
1.71289 & 1.313372 & 0.803134 \\
& 0.01 & 1.772480 & 1.284661 & 0.854901 & 1.772478 & 1.284651 & 0.854699 &
1.772340 & 1.284399 & 0.853806 \\
0.025 & 0.03 & 1.772462 & 1.284610 & 0.854823 & 1.772317 & 1.284299 &
0.852665 & 1.771760 & 1.283154 & 0.846013 \\
& 0.05 & 1.772424 & 1.284502 & 0.854658 & 1.771933 & 1.283467 & 0.848301 &
1.770623 & 1.280410 & 0.829240 \\ \hline\hline
\end{tabular}
}
\end{center}
\end{table}
\begin{figure}[h]
\centering{\small \includegraphics[scale=.20]{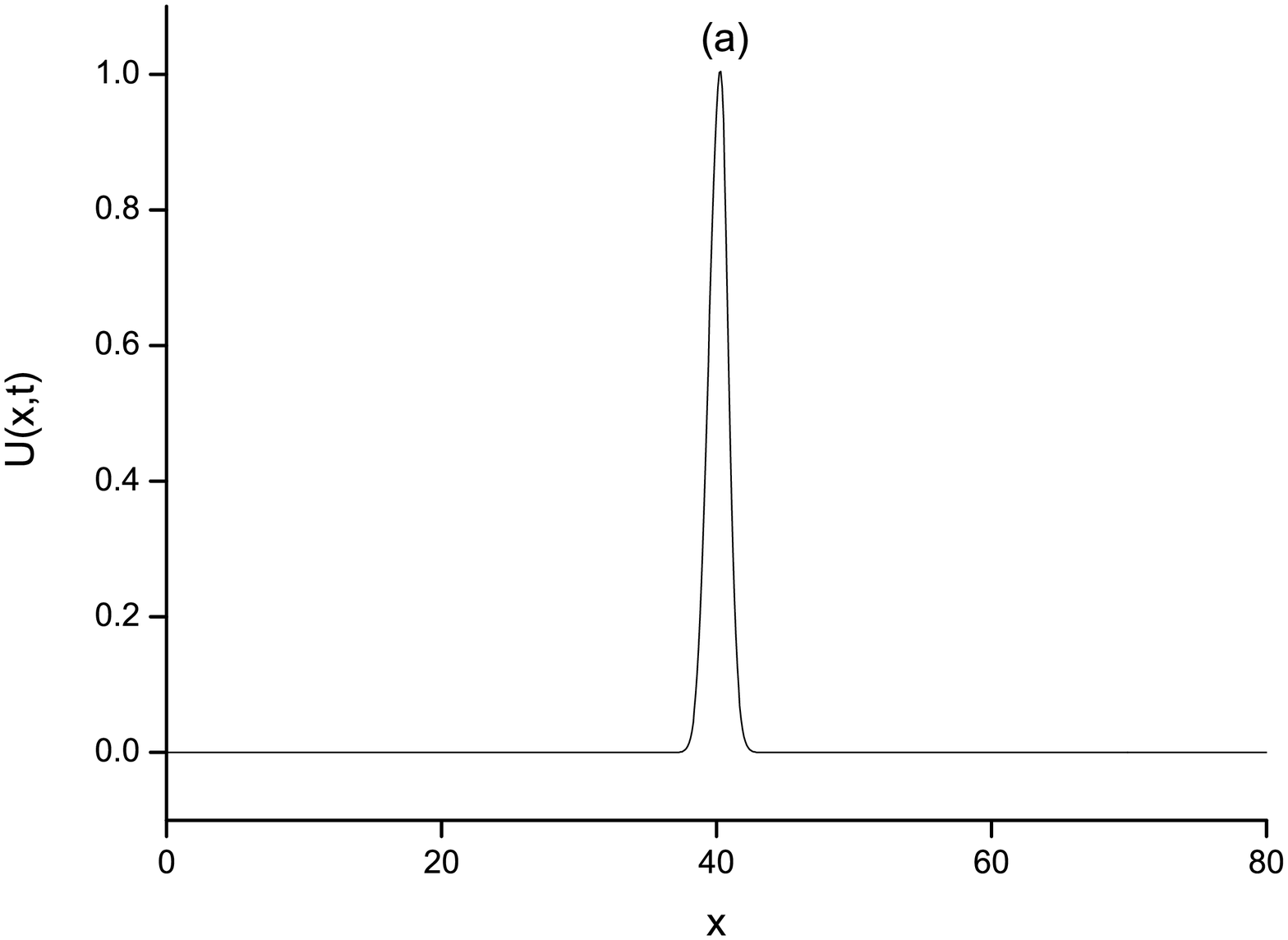} %
\includegraphics[scale=.20]{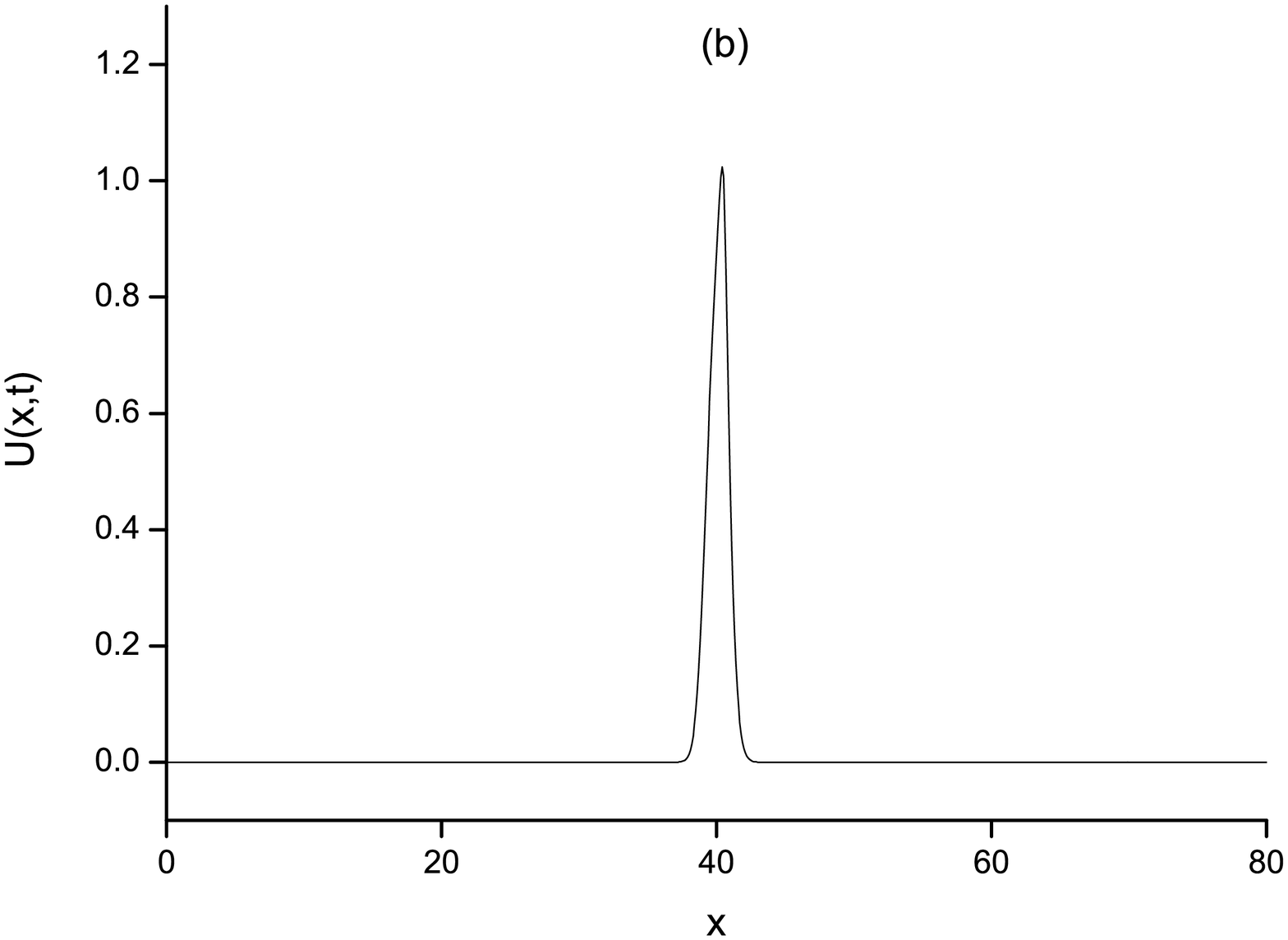}}
\par
{\small \includegraphics[scale=.20]{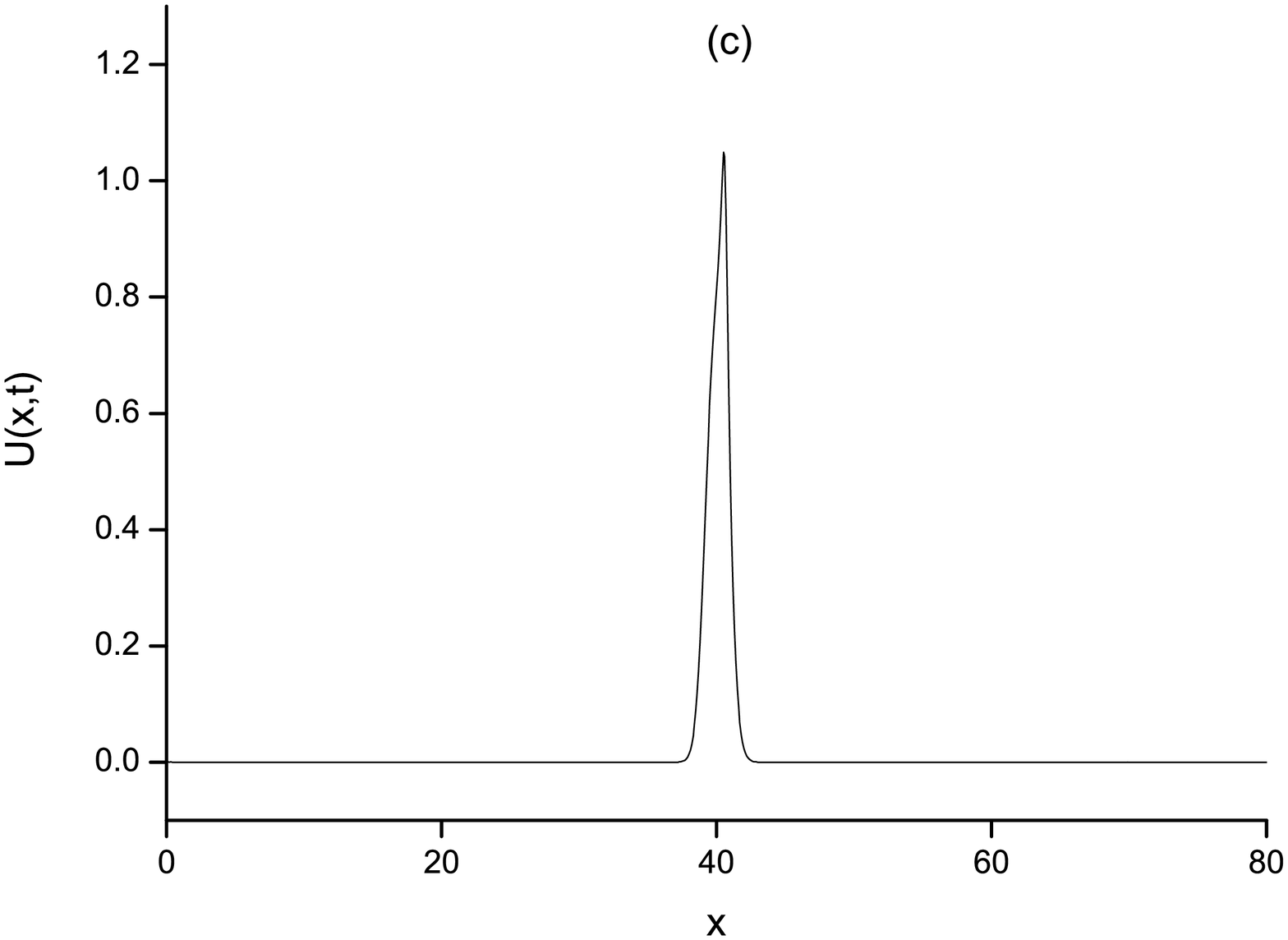}}
\caption{Maxwellian initial condition at $t=0.05$ $(a)$ $p=2,$ $\protect\mu %
=0.025$ $(b)$ $p=3,$ $\protect\mu =0.025$ $(c)$ $p=4,$ $\protect\mu =0.025$}
\label{4055}
\end{figure}

\section{Conclusion}

In this work, a numerical technique based on a Petrov-Galerkin method using
quadratic weight functions and cubic B-spline finite elements has been
proffered to get numerical solutions of GRLW equation. We experimented our
algorithm along with single solitary wave in which the exact solution is
known and broadened it to examine the interaction of two solitary waves and
Maxwellian initial condition where the exact solutions are unknown during
the interaction. Variational formulation and semi-discrete Galerkin scheme
of the equation are generated. Stability analysis have been done and the
linearized numerical scheme have been obtained unconditionally stable. The
accuracy of the method is investigated both $L_{2}$ and $L_{\infty }$ error
norms and the invariant quantities $I_{1}$, $I_{2}$ and $I_{3}$. The
obtained numerical results indicate that the error norms are satisfactorily
small and the conservation laws are marginally constant in all computer
program run. We can see that our numerical scheme for the equation is more
accurate than the other earlier schemes found in the literature. Therefore,
our numerical technique is suitable for\ getting numerical solutions of
partial differential equations.

\end{document}

%% file: analysis_v1.tex
\section{Variational formulation and  energy estimates}
Here we are dedicated to write the initial-boundary value problem  in a variational form, 
and use this weak form to derive some estimates for its solution.  We start by proving  existance and 
uniqueness of solutions by using this variational form. 
The higher order nonlinear  initial boundary value problem \eqref{grlw} can be written as
\begin{equation}
  u_{t}-\mu \Delta u_{t}=\nabla \mathcal{F}(u),~~~~~\label{grlw11}
\end{equation}%
where
$
  \mathcal{F}(u) = -u(1 + p u^{p}),
$
subject to the initial condition
\begin{equation}
   u(x,0)=f_1(x),~~~~~a\leq x\leq b,  \label{intl}
\end{equation}%
and the boundary conditions
\begin{equation}
\begin{array}{lll}
u(a,t)=0,~~~~~u(b,t)=0, &  &  \\
u_{x}(a,t)=0,~~~~~u_{x}(b,t)=0, &  &  \\
u_{xx}(a,t)=0,~~~~~u_{xx}(b,t)=0, & t>0. &
\end{array}
\label{bndry}
\end{equation}

To define the weak form of the solutions of \eqref{grlw11} and to investigate the existence and uniqueness of solutions of the weak
form we define the following spaces.

Here $H^k(\Omega)$, $k\ge 0$ (integer) is an usual normed space of real
valued functions on $\Omega$ and
\begin{equation*}
H_0^{k}(\Omega) = \left\{v \in H^k(\Omega): D^{i}v = 0\ \text{on }
\partial\Omega,\ i = 0, 1, \cdots, k-1 \right\}
\end{equation*}
where $D = \frac{\partial}{\partial x}$. We denote the norm on this space by $\|\cdot \|_k$ which is the usual $H^k$ norm, and  $k=0$ $\|\cdot \|_0 = \|\cdot \|$ represents $L_2$ norm and $(\cdot, \cdot)$ represents $ L_2$  inner product~\cite{Atouani}.

Multiplying \eqref{grlw11} by $\xi\in H_0^1(\Omega)$, and then integrating over $\Omega$ we have
\[ 
  (u_{t}, \xi) - \mu (\Delta u_{t}, \xi) = (\nabla \mathcal{F}(u), \xi).    
\] 
Applying Green's formula on the above inner products we aim to find $u(\cdot, t)\in H_0^1$ so that
\begin{equation}  \label{bbmbur_v2}
\left(u_{t}, \xi \right) + \mu  \left(\nabla u_{t}, \nabla\xi \right) = -\left( \mathcal{F}(u), \nabla \xi \right), \ \forall \ \xi\in H_0^1,
\end{equation}
with $u(0) = u_0$.

\begin{thrm}
\label{thrm01} If $u$ is a solution of \eqref{bbmbur_v2} then
\begin{equation*}
   \|u(t)\|_1 = \|u_0\|_1,\ t\in\ (0,\ T],\
\text{and }\
\|u\|_{L^\infty(L^\infty(\Omega))} \le C\|u_0\|_1
\end{equation*}
holds if $u_0\in H_0^1$, and $C$ is a positive constant.
\end{thrm}
\begin{proof}
Replacing $\xi\in H_0^1 $ by $u \in H_0^1$  in \eqref{bbmbur_v2} results
\begin{equation}\label{bbmbur_v3}
 \left(u_{t}, u  \right) + \mu \left(\nabla u_{t}, \nabla u \right) =
                                              -\left( \mathcal{F}(u), \nabla u  \right)
\end{equation}%
with $u(0) =  u_0$
which gives
\begin{equation}\label{bbmbur_v54}
 \frac{1}{2}\frac{d}{dt} \left[\|u\|^2 + \mu  \|\nabla u\|^2  \right] = \int_{\Omega} u[ \nabla \cdot \mathcal{F}(u)]dx.
\end{equation}
Now
\[
   u \nabla \cdot \mathcal{F}(u) = \nabla \cdot [\mathcal{F}(u)u] -  \nabla \cdot [\mathcal{G}(u)],
\]
if $u \in H_0^1$ where $\mathcal{G}'(u) =  \mathcal{F}(u)$.
For the simplicity of the analysis from now on in this section we fix $\mu = 1$.  The analysis for a general $\mu$ is the similar.
Also, from the initial conditions in \eqref{bbmbur_v3} we have
$ u = 0$ on $\partial \Omega$ and so $\mathcal{G}(0) = 0$, and then
\[
  \int_{\Omega} u[ \nabla \cdot \mathcal{F}(u)]dx = \int_{\Omega} \nabla( u \mathcal{F}(u))dx = 0.
\]
Thus \eqref{bbmbur_v54} takes the form
\[
  \frac{1}{2}\frac{d}{dt} \left(\|u\|_1^2\right) = 0,
\]
and so
\[
 \|u\|_1^2  = \|u_0\|_1^2,
\]
completes the proof of the first part. The second part follows from Sobolev embedding theorem ~\cite{Atouani, T.vider, ciarlet}.
\end{proof}

\begin{thrm}
There exists a unique solution of \eqref{bbmbur_v2} for any $T >0$ such that
\begin{equation*}
u \in L^\infty(0, T, H_0^1(\Omega)) \ \text{with }\ (u(x, 0), \xi) = (u_0,
\xi), \xi\in H_0^1(\Omega),
\end{equation*}
if $u_0 \in H_0^1$ for any $T >0$.
\end{thrm}

\begin{proof}
In order to prove the uniqueness of the solution of \eqref{bbmbur_v2} we consider an orthogonal basis $\{w_i \}_{i = 1}^\infty$  for $H_0^1(\Omega)$
and
$$v^m  = span\{ w_1, w_2, \cdots, w_m\}.$$
Now we define
\[
  u^m(t) = \sum_{i = 1}^m c_i(t) w_i,
\]
for each $t>0$ to satisfy
\begin{equation}\label{bbmbur_v4}
 \left(u_{t}^m, \xi  \right) + \left(\nabla u_{t}^m, \nabla \xi  \right) =
                                              -\left( \mathcal{F} (u^m), \nabla \xi  \right), \ \forall \ \xi\in v^m,
\end{equation}%
  with $u^m(0) =  u_{0, m}$ where
\[
  u_{0, m} = u^m(0) = \sum_{i = 1}^m c_i(0) w_i = P^m u_0.
\]
 Here $P^m$ is an orthogonal projection onto  finite dimensional space $v^m$, and $u_{0, m} \rightarrow u_0\in H_0^1(\Omega)$~\cite{Atouani, ciarlet}. Hence the weak form \eqref{bbmbur_v4} can be written as a system of first order nonlinear ordinary differential equation and there exist a positive time $t_m>0$ such that the nonlinear system of differential equations has a unique solution $u^m$ over $(0,\ t_m)$.

Also from Theorem~\ref{thrm01} it is easy to see that
\[
   \|u^m\|_{\infty} \le C \|u_0\|_1\
\]
 \text{and   }\
\[
  \|\mathcal{F}(u^m)\|^2 \le C \|u_0\|_1^2
 \]
which shows that $\mathcal{F}(u^m)$ is bounded in $L^\infty (0, T, L_2(\Omega))$.  Now by setting
$\xi = u_{t}^m$ in \eqref{bbmbur_v4}
\[    
 \left(u_{t}^m, u_{t}^m  \right) + \left(\nabla u_{t}^m, \nabla u_{t}^m  \right) = -\left( \mathcal{F}(u^m), \nabla u_{t}^m  \right).
\]   
Thus
\[
 \|u_{t}^m\|_1^2 = -\left( \mathcal{F}(u^m), \nabla u_{t}^m  \right)
\]
which yields
\[
  \|u_{t}^m\|_1 \le C \|u_0\|_1.
\]
Hence $\{u^m\}$  and $\{u_{t}^m\}$  are uniformly bounded in $L^{\infty}(0, T, H_0^1(\Omega))$.

By setting
$\xi = w_i$ in \eqref{bbmbur_v4}  we have
\[
 \left(u_{t}^m, w_i  \right) + \left(\nabla u_{t}^m, \nabla w_i  \right) = -\left( \mathcal{F}(u^m), \nabla w_i  \right).
\]
Thus the existence of solutions of the problem follows from the denseness of $\{w_i \}$  in $H_0^1(\Omega)$.

Considering $u$ and $v$ as two solutions of \eqref{bbmbur_v2} with $u(0) = 0$ and $v(0) = 0$, we define  $W = u - v$. Then $W(0) = 0$.
Also
\[
 \left(W_{t}^m, \xi  \right) + \left(\nabla W_{t}^m, \nabla \xi  \right) =
                                              -\left( \mathcal{F} (W^m), \nabla \xi  \right).
\]
Replacing $\xi$ by $W$ in the above equation and following the boundedness of $u$ and $v$ one obtains~\cite{Atouani, T.vider}
\[
 \frac{d}{dt}\|W\|_1 \le C \|W\|_1.
\]
Integrating the above inequality over $[0,\ t]$  yields
\[
   \|W\|_1 \le  \|W(0)\|_1 +  C \int_0^t \|W\|_1 ds.
\]
Now applying Gronwall's Lemma it is easy to see that
\[
  \|W\|_1 \le  e^{Ct} \|W(0)\|_1 = 0,
\]
which confirms $W = 0$ completes the proof~\cite{Atouani, T.vider}.
\end{proof}

\section{ The semidiscrete Galerkin B-spline finite element method}
Consider  $0<h<1$.  A finite dimensional subspace $S_{h}$ of $H_{0}^{1}(\Omega )$ is considered
such that for $u\in H_{0}^{1}(\Omega)\cap H^{4}(\Omega )$, there exists a constant $C$ independent of
 $h$~\cite{Atouani, ciarlet, T.vider} such that
\begin{equation}
\inf_{\xi \in S_{h}}\Vert u-\xi \Vert \leq Ch^{4}.  \label{interp01}
\end{equation}%
Here we aim to find solutions of a semi-discrete finite element formulation of \eqref{grlw11}
$u_{h}:[0,\ T]\rightarrow S_{h}$ such that
\begin{equation}
\left( u_{ht},\xi \right) + \left( \nabla u_{ht},\nabla \xi \right)  =-\left( \mathcal{F} (u_{h}),\nabla \xi \right) ,\ \xi
\in S_{h},  \label{bbmbur_v5}
\end{equation}
with $u_{h}(0)=u_{0,h}\in S_{h}$ is an approximation of $u_{0}$. Before
establishing the original convergence result we first prove a priori bound
of the solution of \eqref{bbmbur_v5} below.

\begin{thrm}
\label{thrm04} The solution $u_h \in S_h$ of \eqref{bbmbur_v5} satisfies
\begin{equation*}
\|u_h\|_1^2  = \|u_{0,h}\|_1^2,\ t\in\
(0,\ T],\
\end{equation*}
\text{and }\
\begin{equation*}
\|u_h\|_{L^\infty(L^\infty(\Omega))} \le C\|u_{0,h}\|_1
\end{equation*}
holds where $C$ is a positive constant.
\end{thrm}

\begin{proof}
  The proof is trivial from our discussion in the previous section (Theorem~\ref{thrm01}).
\end{proof}
Now we move onto establish the theoretical bound of the error in the semi-discrete scheme %
\eqref{bbmbur_v5} of  \eqref{bbmbur_v2}.
%

To that end we consider the following
bilinear form
\begin{equation*}
A(u,v)=(\nabla u,\nabla v),\ \forall\ u,\ v\in H_{0}^{1},
\end{equation*}%
which satisfies the boundedness property
\begin{equation}
|A(u,v)|\leq M\Vert u\Vert _{1}\Vert v\Vert _{1},\forall \ u,\ v\in H_{0}^{1}
\label{boundedness}
\end{equation}%
and coercivity property (on $\Omega $)
\begin{equation}
A(u,u)\geq \alpha \Vert u\Vert _{1},\forall \ u\in H_{0}^{1},\ \text{for
some }\alpha \in \mathbb{R}.  \label{coercivity}
\end{equation}%
Here $A$ satisfies
\begin{equation}
A(u-\tilde{u},\xi )=0,\ \xi \in S_{h},  \label{projection}
\end{equation}%
where $\tilde{u}$ is an auxiliary projection of $u$ \cite{Atouani,
ciarlet,T.vider}.
Now the accuracy result in such a semi-discrete approximation \eqref{bbmbur_v5} of  \eqref{bbmbur_v2} can be
established by the following theorem.
\begin{thrm}
Let $u_h\in S_h$ be a solution of \eqref{bbmbur_v5} and $u\in H_0^1(\Omega)$  be that of  \eqref{bbmbur_v2}, then
the following inequality holds
\begin{equation*}
  \|u - u_h\|\le C h^4,
\end{equation*}
where $C>0$ if $\|u(0) - u_{0, h}\|\le Ch^4$ holds.
\end{thrm}

\begin{proof}
Considering
$
   e = u-u_h
$
we write
$$
  e = \nu +\theta,
\ \quad
\text{   where  } \ \nu = u- \tilde u \ \text{and  } \ \theta = \tilde u - u_h.
$$
%
%
Here
\begin{align*}
\alpha \|u-\tilde u\|_1^2 &\le  A(u-\tilde u, u-\tilde u) \\
                           & =  A(u-\tilde u, u-\xi),\ \xi\in S_h, \text{ from \eqref{coercivity} and \eqref{projection}}.
\end{align*}
Also It follows from
\eqref{boundedness} and \eqref{projection} and \cite{T.vider}  that
\begin{equation}\label{pro_bound}
  \|u-\tilde u\|_1 \le \inf_{\xi\in S_h}\| u - \xi\|_1.
\end{equation}
So \eqref{interp01} and \eqref{pro_bound} confirms the following inequalities
\[
 \|\nu\|_1 \le C h^3 \|u\|_4, \ \text{and  }\ \|\nu \| \le  C h^4 \|u\|_4.
\]
Applying $\frac{\partial}{\partial t}$ on \eqref{projection} and
having some simplifications yields~\cite{T.vider}
\[
   \|\nu_t\|\le  C h^4 \|u_t\|_4.
\]
  Also subtracting  \eqref{bbmbur_v5} from \eqref{bbmbur_v2}  it is easy to see that
\begin{equation}\label{bbmbur_v544}
  (\theta_t, \xi) + (\nabla \theta_t, \nabla\xi)   = - (\nu_t, \xi) - (\mathcal{F}(u)-\mathcal{F}(u_h),\nabla \xi).
\end{equation}
Now substituting  $\xi = \theta$ in \eqref{bbmbur_v544}, and then applying Cauchy-Schwarz inequality  one gets
\[
 \frac{1}{2}\frac{d}{dt} \|\theta\|_1^2  \le    \|\nu_t\|\|\theta\|  +\|\mathcal{F}(u) - \mathcal{F}(u_h)\| \|\nabla \theta\|.
\]
Here
\[
 \|\mathcal{F}(u) - \mathcal{F}(u_h)\| \le C(\|\nu\| + \|\theta\|),
\]
comes from Lipschitz conditions and boundedness of $u$ and $u_h$ and thus
\[
 \frac{d}{dt} \|\theta\|_1^2 \le   C\left( \|\nu_t\|^2 +  \|\nu\|^2  + \|\theta\|^2 + \|\nabla \theta\|^2 \right).
\]
So
\[
 \|\theta\|_1^2 \le \|\theta(0)\|_1^2+  C\int_0^t \left( \|\nu_t\|^2 +  \|\nu\|^2  + \|\theta\|^2 + \|\nabla \theta\|^2 \right)dt.
\]
Hence Gronwall's lemma, bounds of $\nu$ and $\nu_t$ confirms
\[
  \|\theta\|_1 \le C(u) h^4,
\]
if $\theta(0) = 0$, completes the proof~\cite{T.vider, ciarlet}.
\end{proof}